\documentclass[a4paper,reqno]{amsart}

\usepackage[utf8]{inputenc}
\usepackage[english]{babel}
\usepackage{microtype} 

\usepackage{hyperref} 
\hypersetup{hidelinks}
\usepackage[nameinlink,sort]{cleveref} 

\usepackage{orcidlink}

\usepackage{amsmath,amssymb,amsthm}
\usepackage{thmtools} 
\usepackage{mathtools} 
\usepackage{upgreek}
\usepackage{accents} 

\usepackage{tikz} 
\usetikzlibrary{cd}
\usetikzlibrary{calc}

\usepackage{multirow}

\usepackage{dsfont} 
\usepackage{pifont} 
\usepackage{nicefrac} 
\usepackage{nicematrix} 

\usepackage{ifthen}
\usepackage{etoolbox}
\usepackage{enumitem}
\setlist[itemize]{labelindent=\parindent,leftmargin=*,itemsep=3pt,topsep=5pt}
\setlist[enumerate]{label=\textup{(\roman{*})},itemsep=3pt,labelindent=3pt,leftmargin=*}

\newcommand{\R}{\mathbb{R}}
\newcommand{\N}{\mathbb{N}} 

\newcommand{\C}{\mathbb{C}}

\newcommand{\Lp}[2][]{\mathrm{L}_{#2\ifthenelse{\equal{#1}{}}{}{,#1}}} 
\newcommand{\Lb}{\mathcal{L}_{\mathrm{b}}} 
\newcommand{\Lbw}{\Lb^{\textnormal{w}}} 
\newcommand{\Hspace}{\mathrm{H}}

\newcommand{\Hol}{\mathrm{Hol}}
\newcommand{\Ses}{\mathrm{Ses}_{\mathrm{b}}}

\newcommand{\Cc}[1][\infty]{\mathrm{C}_{\mathrm{c}}^{#1}}

\newcommand{\indicator}{\mathds{1}} 
\newcommand{\iu}{\mathrm{i}} 
\newcommand{\euler}{\mathrm{e}} 
\newcommand{\dd}{\mathrm{d}} 
\newcommand{\dx}[1][x]{\,\dd#1}

\newcommand{\conj}[1]{\overline{#1}}

\newcommand{\ball}{\mathrm{B}}

\DeclareMathOperator{\ran}{ran}

\DeclareMathOperator{\dom}{dom}
\DeclareMathOperator{\absbd}{s_b}

\DeclareMathOperator{\Div}{div}
\renewcommand{\div}{\Div}
\DeclareMathOperator{\divn}{\mathring{div}}
\DeclareMathOperator{\grad}{grad}
\DeclareMathOperator{\gradn}{\mathring{grad}}
\DeclareMathOperator{\rot}{curl}
\DeclareMathOperator{\rotn}{\mathring{curl}}

\DeclareMathOperator{\rg}{\mathfrak{g}}
\DeclareMathOperator{\rgn}{\mathfrak{g}_0}
\DeclareMathOperator{\rr}{\mathfrak{c}}
\DeclareMathOperator{\rrn}{\mathfrak{c}_0}

\DeclarePairedDelimiter{\set}{\{}{\}}
\DeclarePairedDelimiter{\norm}{\lVert}{\rVert}
\DeclarePairedDelimiter{\abs}{\vert}{\vert}

\DeclarePairedDelimiterX{\dset}[2]{\{}{\}}{#1\,\delimsize\vert\,\mathopen{} #2}
\DeclarePairedDelimiterX{\scprod}[2]{\langle}{\rangle}{#1,#2}
\DeclarePairedDelimiterX{\dualprod}[2]{\langle}{\rangle}{#1,#2}
\DeclarePairedDelimiterX{\sdprod}[2]{\llangle}{\rrangle}{#1,#2} 

\renewcommand{\Re}{\operatorname{Re}}
\renewcommand{\Im}{\operatorname{Im}}

\newcommand{\cl}[2][]{\overline{#2}\ifthenelse{ \equal{#1}{} }{}{^{#1}}} 

\newcommand{\weakto}{\rightharpoonup}

\theoremstyle{plain}
\newtheorem{theorem}{Theorem}[section]
\newtheorem{lemma}[theorem]{Lemma}

\newtheorem{corollary}[theorem]{Corollary}

\newtoggle{endwithsymbol}
\toggletrue{endwithsymbol}

\newcommand{\myqedhere}{\iftoggle{endwithsymbol}{\qedhere}{}}

\iftoggle{endwithsymbol}{%
    \declaretheorem[style=definition,sibling=theorem,qed=\ding{170}]{definition}
    \declaretheorem[style=definition,sibling=theorem,qed=\ding{169}]{example}
    \declaretheorem[style=definition,sibling=theorem,qed=\ensuremath{\heartsuit}]{problem}
    \declaretheorem[style=definition,sibling=theorem,qed=\ensuremath{\heartsuit}]{assumption}
    \declaretheorem[style=definition,numbered=no,qed=\ensuremath{\heartsuit}]{claim}

    \declaretheorem[style=remark,sibling=theorem,qed=\ding{169}]{remark}
  }%
  {%
    \theoremstyle{definition}
    \newtheorem{definition}[theorem]{Definition}
    \newtheorem{example}[theorem]{Example}

    \theoremstyle{remark}
    \newtheorem{remark}[theorem]{Remark}
  }%

\title[Operator Continuity for Evolutionary Equations]{Weak Operator Continuity for Evolutionary Equations}

\date{\today}

\author[A.~Buchinger]{Andreas Buchinger\,\orcidlink{0009-0004-4203-5874}}
\email{andreas.buchinger@math.tu-freiberg.de}

\author[N.~Skrepek]{Nathanael Skrepek\,\orcidlink{0000-0002-3096-4818}}
\email{nathanael.skrepek@math.tu-freiberg.de}

\author[M.~Waurick]{Marcus Waurick\,\orcidlink{0000-0003-4498-3574}}
\email{marcus.waurick@math.tu-freiberg.de}

\address{TU Bergakademie Freiberg \\
  Institute of Applied Analysis \\
  Akademiestrasse 6 \\
  D-09596 Freiberg \\
  Germany}

\dedicatory{Dedicated to Fritz Gesteszy on the occasion of his 70th birthday}

\begin{document}

\begin{abstract}
 Considering evolutionary equations in the sense of Picard, we identify a certain topology for material laws rendering the   solution operator continuous if considered as a mapping from the material laws into the set of bounded linear operators, where the latter are endowed with the weak operator topology. The topology is a topology of vector-valued holomorphic functions and provides a lift of the previously introduced nonlocal $\mathrm{H}$-topology to particular holomorphic functions. The main area of applications are nonlocal homogenisation results for coupled systems of time-dependent partial differential equations. A continuous dependence result for a nonlocal model for cell migration is also provided.
\end{abstract}

\maketitle

\section{Introduction}%
\label{sec:introduction}

Evolutionary Equations in the sense of Picard provide a unified Hilbert space framework to address well-posedness, regularity, long-time behaviour and other qualitative as well as quantitative properties of predominantly time-dependent partial differential equations. The origins date back to the seminal papers \cite{Pi09,PM11}; several examples can be found in \cite{PMTW20}. We particularly refer to the monograph \cite{SeTrWa22} for a self-contained round-up of the theory. In the linear and autonomous case, evolutionary equations take the form
\begin{equation*}
   (\partial_t M(\partial_t)+A)U=F,
\end{equation*}
formulated in some space-time Hilbert space of functions $\R \to \mathcal{H}$, where $\mathcal{H}$ is a Hilbert space modelling the spatial dependence. $F$ is a given right-hand side, $U$ is the unknown, $A$ is a skew-selfadjoint linear operator, containing the spatial derivatives, and $\partial_t$ is an operator-realisation of the time-derivative. The main conceptual difference to other more traditional approaches of addressing time-dependent partial differential equations is the \emph{material law operator}
\begin{equation*}
   M(\partial_t),
\end{equation*}
which in turn is a bounded, operator-valued holomorphic function $M$ defined on a certain right-half plane of the complex numbers applied to the time-derivative in a sense of a functional calculus. As the name suggests, the material law operator encodes the underlying material properties, i.e., the constitutive relations and their corresponding material coefficients. As a consequence, the complexity of the material is \emph{not} contained in $A$ but rather in $M(\partial_t)$. This has certain advantages. For instance, domain issues or interface phenomena can be more aptly dealt with in the framework of evolutionary equations. Since the material properties are encoded in $M(\partial_t)$, it is of no surprise that homogenisation problems can be reformulated in the framework of evolutionary equations. The main question for these problems is: given a sequence of material laws $(M_n)_n$ does there exist a material law $M$ such that
\begin{equation*}
  (\partial_t M_n(\partial_t)+A)^{-1}\to   (\partial_t M(\partial_t)+A)^{-1}
\end{equation*}
in a suitable (operator) topology as $n\to\infty$. In this general situation, it is possible to answer the question in the affirmative, see \cite{BuErWaUnpub}. Indeed, based on a result for (abstract) Friedrichs systems \cite{BV14,BEW2023}, one can show that such a material law $M$ exists if the above convergence is assumed in the weak operator topology, arguably a very weak assumption. In applications, this result might not be precise enough for the identification of $M$.

Asking more of the operator $A$, we will define a topology on the set of material laws such that if $M_n\to M$ in this topology, then the solution operators converge in the weak operator topology. Since it is possible to show that suitably bounded material laws are compact under this topology, one can also recover a special case of the main result in \cite{BuErWaUnpub} in this case. However, the upshot of the present article is the identification of the topology. The topology introduced is an appropriate generalisation of the Schur topology or nonlocal $\mathrm{H}$-topology as introduced in \cite{Wa22}. Since, in applications, this topology is the precise operator-theoretic description of the topology induced by $\mathrm{H}$-convergence (see \cite{T09,MT97} for $\mathrm{H}$-convergence and \cite{Wa18} for its connections to the nonlocal case) known homogenisation results can be used to find the limit of $M_n$ and, immediately, homogenisation results for time-dependent problems can be obtained. Moreover, note that, in a certain way, the compactness result of the present article can also be viewed as a variant of \cite[Lemma 10.10]{T09} as it asserts that (nonlocal) $\mathrm{H}$-convergence preserves holomorphic dependence on a parameter in the limit.

The idea of finding an appropriate topology on the set of operator-valued holomorphic functions to describe homogenisation problems goes back to the PhD-Thesis \cite{W11}. It has since then been applied to ordinary differential equations (with infinite-dimensional state spaces) \cite{W12,W14}, to partial differential equations \cite{W13,W14a}, or to systems of partial differential equations \cite{W16}. The main assumptions imposed on $A$ were either $A=0$, $A$ being invertible with compact resolvent or, more generally, $\dom(A)\cap \ker(A)^{\perp}$ being compactly embedded in $\mathcal{H}$. An even more general assumption on $A$ was treated in \cite{W16} with severe technical challenges and seemingly undue restrictions on the admissible set of material laws. Here, we may entirely drop all these additional assumptions on the material law imposed in \cite{W16} and thus provide the recently most general form of continuous dependence results for solution operators of evolutionary equations depending on their material laws. We emphasise that this theorem also supersedes the findings for the autonomous case in \cite{W16h}. The present result is also easier to apply compared to previous versions. The approach particularly works for systems of partial differential equations and nonlocal material laws.

We quickly outline how the paper is organised. In \Cref{sec:nHc} we briefly recall the concept of nonlocal $\mathrm{H}$-convergence and discuss its connection to the classical notion of $\mathrm{H}$-convergence.
The concept of evolutionary equations is introduced in \Cref{sec:EEs}. In this section, the notion of material laws and material law operators is properly defined and the fundamental result -- Picard's theorem -- is recalled. \Cref{sec:cluwop} describes the topology of compact convergence for operator-valued holomorphic functions with respect to the weak operator topology. An analogue of the Banach-Alaoglu theorem (\Cref{th:alaoglu-for-holomorphic-linear-operators}) is presented and proved in a comprehensive manner filling some details missing in the (sketch) proofs in \cite[Theorem 4.3]{W14a} or \cite[Theorem 3.4]{W12}. The developed results are used to define the desired topology that characterises nonlocal $\mathrm{H}$-convergence for material laws in \Cref{sec:holwop}, where we also establish a compactness result of said topology. The main theorem of the present contribution is contained in \Cref{sec:AEEs}, where we provide the continuity statement for the solution operator as a mapping from material laws taking values in the bounded linear operators in space-time endowed with the weak operator topology. \Cref{sec:exs} is devoted to two examples: one is about a nonlocal model for cell migration and the other one is concerned with a (nonlocal) homogenisation problem for (scalar) piezo-electricity. We conclude our findings in \Cref{sec:con} and recall some known results in the \Cref{app:app}.
\par\bigskip
All scalar products considered are linear in the first component and antilinear in the second.

\section{Nonlocal $\mathrm{H}$-Convergence}\label{sec:nHc}

The present section is devoted to summarise the rationale introduced in \cite{Wa18}, exemplified in \cite{NW22}, and further studied in \cite{Wa22}. More precisely, we let $\mathcal{H}$ be a Hilbert space and let $\mathcal{H}_0\subseteq \mathcal{H}$ be a closed subspace; $\mathcal{H}_1\coloneqq \mathcal{H}_0^\bot$. Then, any bounded linear operator $M\in \Lb(\mathcal{H})$ can be represented as a $2$-by-$2$ operator matrix
\begin{equation*}
  M = \begin{pmatrix} M_{00}& M_{01} \\ M_{10} & M_{11}\end{pmatrix},
\end{equation*}
where $M_{ij}\in \Lb(\mathcal{H}_j,\mathcal{H}_i)$, $i,j\in \{0,1\}$ (cf.~\Cref{le:components-holomorphic}). We define
\begin{align*}
  \mathcal{M}(\mathcal{H}_{0},\mathcal{H}_{1})
  &\coloneqq \dset*{M \in \Lb(\mathcal{H})}{M_{00}^{-1} \in \Lb(\mathcal{H}_{0}), M^{-1} \in \Lb(\mathcal{H})} \\
  \mathcal{M}(\alpha)
  &\coloneqq
    \begin{aligned}[t]
      \left\{\vphantom{\frac{1}{\alpha_{1}}}M \in \mathcal{M}(\mathcal{H}_{0},\mathcal{H}_{1})
      \,\right|\,&
             \Re M_{00} \geq \alpha_{00}, \Re M_{00}^{-1} \geq \frac{1}{\alpha_{11}}, \\
           &\quad\norm{M_{10} M_{00}^{-1}} \leq \alpha_{10}, \norm{M_{00}^{-1}M_{01}} \leq \alpha_{01}, \\
           &\qquad\Re (M_{11} - M_{10}M_{00}^{-1}M_{01})^{-1} \geq \frac{1}{\alpha_{11}}, \\
           &\left.\vphantom{\frac{1}{\alpha_{1}}}\quad\qquad\mspace{5mu}\Re (M_{11} - M_{10}M_{00}^{-1}M_{01}) \geq \alpha_{00}
      \right\},
    \end{aligned}
\end{align*}
where $\alpha=(\alpha_{ij})_{i,j\in \set{0,1}}\in (0,\infty)^{2\times 2}$. In applications, see \cite{Wa18,NW22,Wa22}, \Cref{ex:Htop} or \Cref{subsec:hompie} below, the decomposition assumed for $\mathcal{H}$ is drawn from the Helmholtz decomposition. This then leads to an appropriate generalisation of $\mathrm{H}$-convergence (or $\mathrm{G}$-convergence) to a general operator-theoretic and possibly but not necessarily nonlocal setting (see \Cref{thm:mrCVPDE}). The definition of nonlocal $\mathrm{H}$-convergence reads as follows.

\begin{definition}\label{def:nHc}
  The \emph{nonlocal $\mathrm{H}$-topology} or \emph{Schur topology}, $\tau(\mathcal{H}_0,\mathcal{H}_1)$, on $\mathcal{M}(\mathcal{H}_0,\mathcal{H}_1)$ is given as the initial topology given by the mappings
  \begin{align*}
    \mathcal{M}(\mathcal{H}_0,\mathcal{H}_1)\ni M &\mapsto M_{00}^{-1} \in \Lbw(\mathcal{H}_0) \\
    \mathcal{M}(\mathcal{H}_0,\mathcal{H}_1)\ni M &\mapsto M_{10}M_{00}^{-1} \in \Lbw(\mathcal{H}_0,\mathcal{H}_1)\\
    \mathcal{M}(\mathcal{H}_0,\mathcal{H}_1)\ni M &\mapsto M_{00}^{-1}M_{01} \in \Lbw(\mathcal{H}_1,\mathcal{H}_0)\\
    \mathcal{M}(\mathcal{H}_0,\mathcal{H}_1)\ni M &\mapsto M_{11}-M_{10}M_{00}^{-1}M_{01} \in \Lbw(\mathcal{H}_1),
  \end{align*}
  where $\Lbw(\mathcal{X},\mathcal{Y})$ denotes the space $\Lb(\mathcal{X},\mathcal{Y})$ endowed with the weak operator topology.
\end{definition}

For later use and illustrational purposes, we quickly provide a sufficient condition for convergence in the topology just introduced. The main tool for the proof is a modification of the proof of \cite[Prop.~13.1.4]{SeTrWa22} (see also \Cref{le:sot-inverse-sequence-conv-crit}).

\begin{lemma}\label{le:sot-conv-implies-nonlocal-h-pointw}
  Let $(M_{n})_{n\in\N}$ be a sequence in $\mathcal{M}(\alpha)$ that converges to $M \in \Lb(\mathcal{H})$ in the strong operator topology. Then, $M \in \mathcal{M}(\alpha)$ and $(M_{n})_{n\in\N}$ converges to $M$ in $\hyperref[def:nHc]{\tau(\mathcal{H}_0,\mathcal{H}_1)}$.
\end{lemma}

\begin{proof}
  First, note that the Pythagorean theorem implies $M_{n,00}\varphi \to M_{00}\varphi$
  and $M_{n,10}\varphi \to M_{10}\varphi$ for all $\varphi \in \mathcal{H}_{0}$, as well as
  $M_{n,01}\psi \to M_{01}\psi$ and $M_{n,11}\psi \to M_{11}\psi$ for all $\psi\in\mathcal{H}_{1}$.

  Since $M_{n} \in \mathcal{M}(\alpha)$ holds, we have $\Re M_{n,00} \geq \alpha_{00}$  for $n\in\N$.
  Together with $M_{n,00}\varphi \to M_{00}\varphi$ for $\varphi \in \mathcal{H}_{0}$, this shows
\begin{equation*}
  \Re M_{00} \geq \alpha_{00}\text{.}
\end{equation*}
  Considering \Cref{le:realpart-bounded-from-below-implies-invertible},
  we infer that $M_{00}$ is boundedly invertible and that the sequence $(M_{n,00}^{-1})_{n\in\N}$
  is uniformly bounded in the operator norm. Hence, we can apply \Cref{le:sot-inverse-sequence-conv-crit}
  to prove
\begin{equation*}
  M_{n,00}^{-1}\varphi \to M_{00}^{-1}\varphi
\end{equation*}
  for all $\varphi \in \mathcal{H}_{0}$.
  We also know that
  $\Re M_{n,00}^{-1} \geq 1/\alpha_{11}$ for all $n\in\N$, which now implies
\begin{equation*}
  \Re M_{00}^{-1} \geq 1/\alpha_{11}\text{.}
\end{equation*}

  Addition and multiplication of operators are
  sequentially continuous w.r.t.\ the strong operator topology. Thus, we immediately
  obtain
  \begin{align*}
    M_{n,10} M_{n,00}^{-1}\varphi&\to M_{10} M_{00}^{-1}\varphi\text{,}\\
    M_{n,00}^{-1}M_{n,01}\psi&\to M_{00}^{-1}M_{01}\psi\text{ and}\\
    M_{n,11}\psi - M_{n,10}M_{n,00}^{-1}M_{n,01}\psi&\to M_{11}\psi - M_{10}M_{00}^{-1}M_{01}\psi
  \end{align*}
  for all $\varphi \in \mathcal{H}_{0}$ and all $\psi\in\mathcal{H}_{1}$.
  Once again, \Cref{le:realpart-bounded-from-below-implies-invertible} and \Cref{le:sot-inverse-sequence-conv-crit} yield
  \begin{align}
    \Re (M_{11} - M_{10}M_{00}^{-1}M_{01}) &\geq \alpha_{00}\text{ and}\label{eq:inv-cond-block-matrix}\\
    \Re (M_{11} - M_{10}M_{00}^{-1}M_{01})^{-1} &\geq 1/\alpha_{11}\text{.}\notag
  \end{align}

Hence, \Cref{le:realpart-bounded-from-below-implies-invertible} even allows us to explicitly write down the
  inverse of $M$, which we will omit here (cf., however, \Cref{le:inverse-of-t-plus-a}). This means $M^{-1}\in \Lb(\mathcal{H})$.

  Clearly, operator norm balls are closed in the strong operator topology. Hence, we lastly infer
  \begin{equation*}
    \norm{M_{10} M_{00}^{-1}} \leq \alpha_{10}\quad\text{ and }\quad
    \norm{M_{00}^{-1}M_{01}}\leq \alpha_{01}\text{.}\qedhere
  \end{equation*}
\end{proof}

The reason for having introduced the \hyperref[def:nHc]{nonlocal $\mathrm{H}$-topology} is the consideration of homogenisation problems in an operator-theoretic setting. For this, we quickly recall a standard situation for choices of $\mathcal{H}_0$ and $\mathcal{H}_1$, respectively.
\begin{example}\label{ex:Htop}
  Let $\Omega\subseteq\R^3$ a bounded weak Lipschitz domain with continuous boundary.
  Then,
  \begin{align*}
    \rgn & \coloneqq \dset{\grad u}{u\in \Hspace_0^1(\Omega)}, \\
    \rr & \coloneqq \dset{\rot E}{E\in \Lp{2}(\Omega)^3, \rot E\in \Lp{2}(\Omega)^3}, \\
    \rg & \coloneqq \dset{\grad u}{u\in \Hspace^1(\Omega)},\quad\text{and} \\
    \rrn & \coloneqq \dset{\rot E}{\exists (E_n)_n \in \Cc(\Omega)^3\colon E_n\to E\in \Lp{2}(\Omega)^3, \rot E_n\to \rot E \in \Lp{2}(\Omega)^3}
  \end{align*}
  are closed subspaces of $\Lp{2}(\Omega)^3$.
  Indeed, see~\cite{PZ20} for the FA-toolbox establishing closed range results based on selection theorems (see also~\cite[Lemma~4.1]{ElstGordWau2019} for the technique) and~\cite{BPS16,P84} for the Picard--Weber--Weck selection theorem needed.

  Then, the orthogonal decompositions
  \begin{equation*}
    \Lp{2}(\Omega)^3= \rgn\oplus \rr \oplus \mathcal{H}_D(\Omega) = \rg\oplus \rrn \oplus \mathcal{H}_N(\Omega),
  \end{equation*}
  hold, see, e.g., \cite{PW22}, where
  $\mathcal{H}_D(\Omega)$ and $\mathcal{H}_N(\Omega)$ are finite-dimensional subspaces, whose dimensions can be found to be the number of bounded connected components and the number of handles respectively, see \cite{P82,PW22} for the details.
\end{example}

One of the main results of \cite{Wa18} is the identification of the \hyperref[def:nHc]{nonlocal $\mathrm{H}$-topology} as the precise topology describing homogenisation problems. For this, we introduce for $\Omega\subseteq \R^3$ open and $0<\alpha<\beta$
\begin{equation*}
  M(\alpha,\beta;\Omega)
  \coloneqq \dset*{a\in \Lp{\infty}(\Omega)^{3\times 3}}{\Re a(x)\geq \alpha, \Re a(x)^{-1}\geq 1/\beta \ (\text{a.e.}\ x\in \Omega)}.
\end{equation*}
Next, we can identify $M(\alpha,\beta;\Omega)$ as a subspace of $\Lb(\Lp{2}(\Omega)^3)$ and, thus, endow it with the trace topology induced by the \hyperref[def:nHc]{nonlocal $\mathrm{H}$-topology} subject to the decompositions exemplified in \Cref{ex:Htop}. This straightforwardly works for $\Omega$ being \emph{topologically trivial}; that is, for $\mathcal{H}_D(\Omega)= \mathcal{H}_N(\Omega)=\{0\}$. For topologically non-trivial domains, the situation is more involved. This is dealt with in~\cite[Section 5]{Wa22}.
The main result of~\cite{Wa18} shows that \hyperref[def:nHc]{nonlocal $\mathrm{H}$-convergence} indeed generalises $\mathrm{H}$-convergence appropriately and reads as follows.

\begin{theorem}\label{thm:mrCVPDE}
  Let $\Omega\subseteq \R^3$ be a bounded weak Lipschitz domain with continuous boundary. Additionally, assume $\Omega$ to be topologically trivial. Let $0<\alpha<\beta$ and $(a_n)_n$ in $M(\alpha,\beta;\Omega)$, $a\in \Lb(\Lp{2}(\Omega)^{3\times 3})$. Then, the following conditions are equivalent:
  \begin{enumerate}
    \item $a\in M(\alpha,\beta;\Omega)$ and $(a_n)_n$ \emph{$\mathrm{H}$-converges to} $a$; that is, for all $f\in \Hspace^{-1}(\Omega)$ and $u_n\in \Hspace_0^1(\Omega)$ satisfying
          \begin{equation*}
          \scprod{a_n \grad u_n}{\grad \varphi} = f(\varphi)\quad(\varphi\in \Hspace_{0}^{1}(\Omega)),
          \end{equation*}
          we have $u_n\weakto u\in \Hspace_{0}^{1}(\Omega)$ and $a_n\grad u_n\weakto a\grad u \in \Lp{2}(\Omega)^{3}$, where
          \begin{equation*}
          \scprod{a \grad u}{\grad \varphi} = f(\varphi)\quad(\varphi\in \Hspace_{0}^{1}(\Omega));
          \end{equation*}
    \item $(a_n)_n$ converges to $a$ in $\hyperref[def:nHc]{\tau(\rgn,\rr)}$;
    \item $(a_n)_n$ converges to $a$ in $\hyperref[def:nHc]{\tau(\rg,\rrn)}$.
  \end{enumerate}
\end{theorem}

\section{Evolutionary Equations}\label{sec:EEs}

In this section, we briefly summarise the concept of evolutionary equations as introduced by Picard, see \cite{Pi09}; we particularly refer to \cite{SeTrWa22} for a recent monograph on the subject matter. For $\nu>0$ and a Hilbert space $\mathcal{H}$ we define
\begin{equation*}
  \Lp[\nu]{2}(\R;\mathcal{H})\coloneqq \dset*{f\in \Lp[\textup{loc}]{1}(\R;\mathcal{H})}{ \int_\R \norm{f(s)}_{\mathcal{H}}^2 \exp(-2\nu s)\dx[s] <\infty}.
\end{equation*}
The distributional derivative, $\partial_t$, realised as an operator in $\Lp[\nu]{2}(\R;\mathcal{H})$ endowed with the maximal domain is continuously invertible, with explicit inverse given by
\begin{equation*}
   \partial_t^{-1}f(t) = \int_{-\infty}^t f(s) \dx[s]\quad(f\in \Lp[\nu]{2}(\R;\mathcal{H})).
\end{equation*}
We have $\norm{\partial_t}_{\Lb(\Lp[\nu]{2}(\R;\mathcal{H}))}\leq 1/\nu$. The (unitary) Fourier-Laplace transformation, $\mathcal{L}_\nu\colon \Lp[\nu]{2}(\R;\mathcal{H})\to \Lp{2}(\R;\mathcal{H}) $ provides an explicit spectral theorem for $\partial_t$. Indeed, using $\mathrm{m}$ to denote the multiplication-by-argument operator with maximal domain in $\Lp{2}(\R;\mathcal{H})$, we have
\begin{equation*}
  \partial_t = \mathcal{L}_\nu^*(\iu \mathrm{m}+\nu)\mathcal{L}_\nu,
\end{equation*}
where, for compactly supported, continuous $\varphi\colon \R\to \mathcal{H}$,
\begin{equation*}
  \mathcal{L}_\nu \varphi(\xi) = \frac{1}{\sqrt{2\pi}}\int_{\R} \euler^{-(\iu \xi  +\nu)s}\varphi(s)\dx[s] \quad (\xi\in \R).
\end{equation*}
The spectral representation for $\partial_t$ gives rise to a functional calculus. It is enough to consider bounded, holomorphic functions on some right-half planes of $\C$. This leads us to define the notion of material laws and corresponding material law operators:

\begin{definition}\label{def:matlaw}
  We call $M\colon \dom(M)\subseteq \C\to \Lb(\mathcal{H})$ \emph{material law}, if $\dom(M)$ is open, $M$ is holomorphic and if there exists $\nu_0>0$ such that $\C_{\Re>\nu_0}\subseteq \dom(M)$ with
  \begin{equation*}
    \sup_{z\in \C_{\Re>\nu_0}}\norm{M(z)}<\infty\text{.}
  \end{equation*}
  The infimum over all such $\nu_0$ is denoted by $\absbd(M)$, the \emph{abscissa of boundedness}. If $M$ is a material law and $\nu>\absbd(M)$, we define $M(\partial_t) \in \Lb(\Lp[\nu]{2}(\R;\mathcal{H}))$, the \emph{(corresponding) material law operator}, via
  \begin{equation*}
    M(\partial_t)\coloneqq \mathcal{L}_\nu^*M(\iu \mathrm{m}+\nu)\mathcal{L}_\nu,
  \end{equation*}
  where for $\varphi \in \Lp{2}(\R;\mathcal{H})$ we put
  \begin{equation*}
    (M(\iu \mathrm{m}+\nu)\varphi)(\xi)\coloneqq M(\iu \xi +\nu)\varphi(\xi) \quad(\text{a.e.}~\xi\in \R).
    \myqedhere
  \end{equation*}
\end{definition}

The well-posedness theorem for evolutionary equations reads as follows. It is applicable to both classical examples like heat and Maxwell's equations or non-standard ones like time-nonlocal examples from elasticity theory, mixed type equations or equations with coefficients that are nonlocal in space. A closed and densely defined operator in $\mathcal{H}$ can be (canonically) lifted to $\Lp[\nu]{2}(\R;\mathcal{H})$ by pointwise application; this (abstract) multiplication operator will be denoted by the same name. Note that if the initial operator is skew-selfadjoint, so is the lifted one.

\begin{theorem}[{{Picard's Theorem, see, e.g., \cite[Theorem 6.2.1]{SeTrWa22}}}]\label{th:pt}
  Let $\nu>0$, $\mathcal{H}$ a Hilbert space, $A\colon \dom(A)\subseteq \mathcal{H}\to \mathcal{H}$ a skew-selfadjoint operator. Let $M\colon \dom(M)\subseteq \C\to \Lb(\mathcal{H})$ be a material law with $\nu>\absbd(M)$. If, there exists $c>0$ such that for all $z\in \C_{\Re>\nu}$ we have
  \begin{equation*}
    \Re zM(z) \geq c\text{.}
  \end{equation*}
  Then, the operator sum
  \begin{equation*}
    (\partial_t M(\partial_t) + A) \colon \dom(\partial_t M(\partial_t))\cap \dom(A)\subseteq \Lp[\nu]{2}(\R;\mathcal{H})\to \Lp[\nu]{2}(\R;\mathcal{H})
  \end{equation*}
  is closable. Its closure is continuously invertible with
  \begin{equation*}
    \norm{\cl{\partial_t M(\partial_t)+A}^{-1}} \leq \tfrac{1}{c}.
  \end{equation*}
  Moreover, $\mathcal{S}\coloneqq \cl{\partial_t M(\partial_t)+A}^{-1}$ is the material law operator corresponding to the material law
  \begin{equation*}
    z\mapsto (zM(z)+A)^{-1},
  \end{equation*}
  $\mathcal{S}$ is \emph{causal}; that is,
  \begin{equation*}
    \forall a\in \R: \indicator_{(-\infty,a)}\mathcal{S}\indicator_{(-\infty,a)} = \indicator_{(-\infty,a)}\mathcal{S};
  \end{equation*}
  and, if $f\in \dom(\partial_t)$, then $\mathcal{S}f\in \dom(\partial_t) \cap \dom(A)\subseteq \dom(\partial_t M(\partial_t))\cap \dom(A)$.
\end{theorem}

\section{Compactness in the Locally Uniform Weak Operator Topology}\label{sec:cluwop}

In this section we consider the space of holomorphic functions that map into the space of bounded operators between two Hilbert spaces. We endow this space with an initial topology that is motivated by the weak operator topology. The new aspect is that we additionally have holomorphic dependence on a complex parameter.

\begin{lemma}\label{th:G-from-operator-to-sesqui-is-bijective}
  Let $U\subseteq\C$ be open and let $\mathcal{X}$, $\mathcal{Y}$ be Hilbert spaces.
  We regard $\Hol(U,\Lb(\mathcal{X},\mathcal{Y}))$ and $\Hol(U,\Ses(\mathcal{X},\mathcal{Y}))$, where $\Lb(\mathcal{X},\mathcal{Y})$ is endowed with the operator norm and the bounded sesquilinear forms $\Ses(\mathcal{X},\mathcal{Y})$ are endowed with the sesquilinear norm, i.e., $\norm{\sigma} \coloneqq \sup_{\norm{\varphi}=1,\norm{\psi}=1} \abs{\sigma(\varphi,\psi)}$ for
  $\sigma\in \Ses(\mathcal{X},\mathcal{Y})$.
  Then, the mapping
  \begin{align}\label{eq:isomorphism-between-Hol-Lb-and-Hol-Ses}
    G \colon\left\{
    \begin{array}{rcl}
      \Hol(U,\Lb(\mathcal{X},\mathcal{Y})) & \to & \Hol(U,\Ses(\mathcal{X},\mathcal{Y})), \\
      f & \mapsto & z \mapsto \scprod{f(z)\cdot}{\cdot}_{\mathcal{Y}}.
    \end{array}
    \right.
  \end{align}
  is bijective and linear.
\end{lemma}

\begin{proof}
  It is well-known that mapping an operator to the corresponding sesquilinear form is a linear and isometric bijection from $\Lb(\mathcal{X},\mathcal{Y})$ to
  $\Ses(\mathcal{X},\mathcal{Y})$. Hence, $G$ is a linear bijection from $\Lb(\mathcal{X},\mathcal{Y})^{U}$ to $\Ses(\mathcal{X},\mathcal{Y})^{U}$.

  Let $f \in \Hol(U,\Lb(\mathcal{X},\mathcal{Y}))$. Then the complex derivative $f'$ exists. Thus,
  \begin{align*}
    \MoveEqLeft[7]
    \lim_{w \to z} \norm*{\frac{G(f)(z) - G(f)(w)}{z-w} - G(f')(z)}_{\Ses(\mathcal{X},\mathcal{Y})} \\
    &= \lim_{w \to z} \norm*{\frac{G(f)(z) - G(f)(w) - G(f')(z)(z-w)}{z-w}}_{\Ses(\mathcal{X},\mathcal{Y})} \\
    &= \lim_{w \to z} \norm*{G\Big(\frac{f(z) - f(w) - f'(z)(z-w)}{z-w}\Big)}_{\Ses(\mathcal{X},\mathcal{Y})} \\
    &= \lim_{w \to z} \norm*{\frac{f(z) - f(w)}{z-w} - f'(z)}_{\Lb(\mathcal{X},\mathcal{Y})} = 0,
  \end{align*}
  which implies $G(f)\in \Hol(U,\Ses(\mathcal{X},\mathcal{Y}))$ with $G(f)'(z) = G(f')(z)$.
  For $\sigma \in \Hol(U,\Ses(\mathcal{X},\mathcal{Y}))$ we can analogously prove $G^{-1}(\sigma)\in \Hol(U,\Lb(\mathcal{X},\mathcal{Y}))$ with $G^{-1}(\sigma)'(z)=G^{-1}(\sigma')(z)$.
\end{proof}


\begin{definition}\label{def:locally-uniform-weak-operator-topology}
  Let $U \subseteq \C$ be an open set and $\mathcal{X}$, $\mathcal{Y}$ Hilbert spaces. For every $\varphi \in \mathcal{X}$ and $\psi \in \mathcal{Y}$ we define
  (cf.\ \Cref{re:compact-convergence-topology-properties})
  \begin{equation*}
    \Lambda_{\varphi,\psi}\colon
    \left\{
      \begin{array}{rcl}
       \Hol(U,\Lb(\mathcal{X},\mathcal{Y})) & \to & \Hol(U,\C), \\
       f & \mapsto & \scprod{f(\cdot)\varphi}{\psi}_{\mathcal{Y}},
      \end{array}
    \right.
  \end{equation*}
  and
  \begin{equation*}
    \tilde{\Lambda}_{\varphi,\psi}\colon
    \left\{
      \begin{array}{rcl}
       \Hol(U,\Ses(\mathcal{X},\mathcal{Y})) & \to & \Hol(U,\C), \\
       \sigma & \mapsto & \sigma(\cdot)(\varphi,\psi).
      \end{array}
    \right.
  \end{equation*}
  Moreover, we define the \emph{locally uniform weak operator topology} on $\Hol(U,\Lb(\mathcal{X},\mathcal{Y}))$ as the initial topology  w.r.t.\ the mappings $\Lambda_{\varphi,\psi}$ for $\varphi \in \mathcal{X}$ and $\psi \in \mathcal{Y}$ and denote it by $\mathcal{T}_{\Lambda}$. Analogously, we denote the initial topology on $\Hol(U,\Ses(\mathcal{X},\mathcal{Y}))$ w.r.t.\ the mappings $\tilde{\Lambda}_{\varphi,\psi}$ for all $\varphi \in \mathcal{X}$ and $\psi \in \mathcal{Y}$ by $\mathcal{T}_{\tilde{\Lambda}}$.
\end{definition}

\begin{remark}\label{re:loc-unif-wot-is-tvs-and-hausdorff}
The mappings $\Lambda_{\varphi,\psi}$ are linear and $\bigcap\ker\Lambda_{\varphi,\psi}=\{0\}$. Therefore, the corresponding initial topology
  $\mathcal{T}_{\Lambda}$ is Hausdorff and makes $\Hol(U,\Lb(\mathcal{X},\mathcal{Y}))$ a topological vector space.
\end{remark}

\begin{lemma}\label{th:G-from-operator-to-sesqui-is-isomorphism}
  Let $G\colon \Hol(U,\Lb(\mathcal{X},\mathcal{Y})) \to \Hol(U,\Ses(\mathcal{X},\mathcal{Y}))$ be the linear bijection from~\eqref{eq:isomorphism-between-Hol-Lb-and-Hol-Ses}. Then $G \colon (\Hol(U,\Lb(\mathcal{X},\mathcal{Y})),\mathcal{T}_{\Lambda}) \to (\Hol(U,\Ses(\mathcal{X},\mathcal{Y})),\mathcal{T}_{\tilde{\Lambda}})$ is a linear homeomorphism.
\end{lemma}

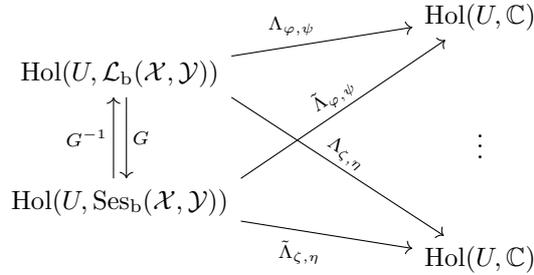
\begin{figure}[h]
  \centering
  \begin{tikzcd}[row sep=1ex,column sep=15ex]
    & \Hol(U,\C) \\
    \Hol(U,\Lb(\mathcal{X},\mathcal{Y})) \ar[shift left]{dd}{G} \ar{ur}{\Lambda_{\varphi,\psi}} \ar[start anchor=south east]{dddr}[sloped]{\Lambda_{\zeta,\eta}} & \\
    & \vdots \\
    \Hol(U,\Ses(\mathcal{X},\mathcal{Y})) \ar[shift left]{uu}{G^{-1}}\ar{dr}[swap]{\tilde{\Lambda}_{\zeta,\eta}} \ar[start anchor=north east]{uuur}[sloped]{\tilde{\Lambda}_{\varphi,\psi}}& \\
    & \Hol(U,\C)
  \end{tikzcd}
  \caption{Initial topology on $\Hol(U,\Lb(\mathcal{X},\mathcal{Y}))$ and $\Hol(U,\Ses(\mathcal{X},\mathcal{Y}))$}
  \label{fig:illustration-initial-topology-on-Hol-spaces}
\end{figure}

\begin{proof}
  By definition of $G$, $\Lambda_{\varphi,\psi}$ and $\tilde{\Lambda}_{\varphi,\psi}$ we can immediately see that $\tilde{\Lambda}_{\varphi,\psi} \circ G = \Lambda_{\varphi,\psi}$ and since $G$ is invertible we also have $\tilde{\Lambda}_{\varphi,\psi} = \Lambda_{\varphi,\psi} \circ G^{-1}$. Hence, the diagram in \Cref{fig:illustration-initial-topology-on-Hol-spaces} commutes. Since initial topologies are transitive, we conclude that $G$ is a homeomorphism.
\end{proof}

\begin{lemma}\label{th:very-weak-holomorphy-for-Lb-and-Ses-implies-holomorphy}
  A function $\sigma\colon U \to \Ses(\mathcal{X},\mathcal{Y})$ is holomorphic if and only if $\sigma(\cdot)(\varphi,\psi)$ is holomorphic for all $\varphi \in \mathcal{X}$ and $\psi \in \mathcal{Y}$.
\end{lemma}

\begin{proof}
  In view of \Cref{th:G-from-operator-to-sesqui-is-bijective} and \Cref{th:weak-is-strong-holomorphy}, it suffices to prove that a set of operators
  $\mathcal{B}\subseteq \Lb(\mathcal{X},\mathcal{Y})$ is bounded (w.r.t.\ the operator norm) if and only if
  \begin{equation}\label{eq:weakly-bounded-set-of-bounded-operators}
    \forall\varphi\in\mathcal{X}\forall\psi\in\mathcal{Y}:\sup_{A\in \mathcal{B}}\abs{\scprod{A\varphi}{\psi}_{\mathcal{Y}}}<\infty\text{.}
  \end{equation}
  Obviously, boundedness of $\mathcal{B}$ implies~\eqref{eq:weakly-bounded-set-of-bounded-operators}.

  Conversely, we can keep $\varphi\in\mathcal{X}$ fixed and write $\iota_{A\varphi}\in\Lb(\mathcal{Y},\C)$ for the functional
  with Riesz representation $A\varphi\in\mathcal{Y}$, i.e., $\psi\mapsto \scprod{\psi}{A\varphi}_{\mathcal{Y}}$.
  Then,~\eqref{eq:weakly-bounded-set-of-bounded-operators} and the uniform boundedness principle yield
  \begin{equation*}
    \forall \psi\in\mathcal{Y}:\sup_{A\in \mathcal{B}}\abs{\iota_{A\varphi}(\psi)}<\infty \implies\sup_{A\in \mathcal{B}}\norm{A\varphi}_{\mathcal{Y}}<\infty\text{.}
  \end{equation*}
  Since this holds true for every $\varphi\in\mathcal{X}$, another iteration of the uniform boundedness principle shows that $\mathcal{B}$ is bounded.
\end{proof}

Similarly to Banach-Alaoglu's theorem (for the weak operator topology) we show that 
\begin{equation*}
K\coloneqq \dset{\sigma \in \Hol(U,\Ses(\mathcal{X},\mathcal{Y}))}{\forall z\in U:\norm{\sigma(z)} \leq 1}
\end{equation*} (the analogue of the closed ball in our setting) is compact.

\begin{theorem}\label{th:alaoglu-for-holomorphic-sesquilinear-forms}
  $K$ is a $\mathcal{T}_{\tilde{\Lambda}}$-compact subset of $\Hol(U,\Ses(\mathcal{X},\mathcal{Y}))$.
  If both $\mathcal{X}$ and $\mathcal{Y}$ are separable, then $(K,\mathcal{T}_{\tilde{\Lambda}})$ is metrisable.
\end{theorem}

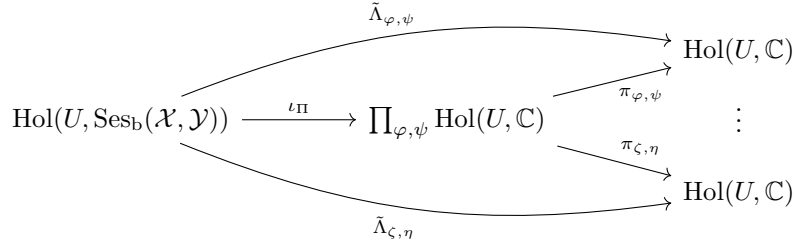
\begin{figure}[h]
  \centering
  \begin{tikzcd}[row sep=2ex,column sep=10ex]
    & & \Hol(U,\C)\\
    \Hol(U,\Ses(\mathcal{X},\mathcal{Y})) \ar{r}{\iota_{\Pi}} \ar[bend left=15]{urr}{\tilde{\Lambda}_{\varphi,\psi}} \ar[bend right=15]{drr}[swap]{\tilde{\Lambda}_{\zeta,\eta}} & \prod_{\varphi,\psi} \Hol(U,\C) \ar{dr}{\pi_{\zeta,\eta}} \ar{ur}[swap]{\pi_{\varphi,\psi}} & \smash{\raisebox{-0.5ex}{\vdots}}\\
    & & \Hol(U,\C)
  \end{tikzcd}
  \caption{Initial topology on $\Hol(U,\Ses(\mathcal{X},\mathcal{Y}))$}
  \label{fig:initial-topology-Hol-Ses-associative}
\end{figure}

\begin{proof}
  We define the following mappings
  \begin{align*}
    \iota_{\Pi}\colon
    \left\{ \renewcommand*{\arraystretch}{1.5}
    \begin{array}{rcl}
     \Hol(U,\Ses(\mathcal{X},\mathcal{Y})) & \to & \displaystyle\prod_{(\varphi,\psi) \in \mathcal{X} \times \mathcal{Y}} \Hol(U,\C), \\
     \sigma & \mapsto & \big(\sigma(\cdot)(\varphi,\psi)\big)_{(\varphi,\psi) \in \mathcal{X} \times \mathcal{Y}},
    \end{array}
    \right.
  \end{align*}
  and
  \begin{align*}
    \pi_{\zeta,\eta}\colon
    \left\{ \renewcommand*{\arraystretch}{1.5}
    \begin{array}{rcl}
     \displaystyle\prod_{(\varphi,\psi) \in \mathcal{X} \times \mathcal{Y}} \Hol(U,\C) & \to & \Hol(U,\C), \\
     \big(f(\cdot)(\varphi,\psi)\big)_{(\varphi,\psi) \in \mathcal{X} \times \mathcal{Y}} & \mapsto & f(\cdot)(\zeta,\eta).
    \end{array}
    \right.
  \end{align*}
  We can immediately see that $\tilde{\Lambda}_{\varphi,\psi} = \pi_{\varphi,\psi} \circ \iota_{\Pi}$.
  Hence, the diagram in \Cref{fig:initial-topology-Hol-Ses-associative} is commutative. If we endow $\prod_{(\varphi,\psi)} \Hol(U,\C)$ with the product topology, i.e., the initial
  topology w.r.t.\ the $\pi_{\zeta,\eta}$ and $\iota_{\Pi}(\Hol(U,\Ses(\mathcal{X},\mathcal{Y})))\subseteq \prod_{(\varphi,\psi)} \Hol(U,\C)$ with the trace topology, the transitivity
  of initial topologies and the commutativity of the diagram imply that $\iota_{\Pi}$ is a homeomorphism onto its range.

  Therefore, $K$ is compact in $\Hol(U,\Ses(\mathcal{X},\mathcal{Y}))$, if and only if $\iota_{\Pi}(K)$ is compact in $\prod_{(\varphi,\psi)}\Hol(U,\C)$.

  First, we show that $\iota_{\Pi}\big(\Hol(U,\Ses(\mathcal{X},\mathcal{Y}))\big)$ is closed in $\prod_{(\varphi,\psi)} \Hol(U,\C)$:
  Let $(\sigma_{i})_{i \in I}$ be a net in $\Hol(U,\Ses(\mathcal{X},\mathcal{Y}))$ such that $(\iota_{\Pi} \sigma_{i})_{i\in I}$ converges to \linebreak $f\in \prod_{(\varphi,\psi)} \Hol(U,\C)$ (w.r.t.\ the product topology). The
  sesquilinearity of the $\sigma_{i}$ exactly means
  \begin{multline*}
    \sigma_{i}(\cdot)(\varphi_{1} + \alpha \varphi_{2},\psi_{1} + \beta \psi_{2}) \\
    = \sigma_{i}(\cdot)(\varphi_{1},\psi_{1})
    + \alpha \sigma_{i}(\cdot)(\varphi_{2},\psi_{1}) + \conj{\beta}\sigma_{i}(\cdot)(\varphi_{1},\psi_{2})
    + \alpha \conj{\beta} \sigma_{i}(\cdot)(\varphi_{2},\psi_{2})
  \end{multline*}
  or equivalently
  \begin{align*}
    \pi_{\varphi_{1}+\alpha\varphi_{2},\psi_{1}+\beta\psi_{2}}\iota_{\Pi}\sigma_{i}
    = \pi_{\varphi_{1},\psi_{1}} \iota_{\Pi} \sigma_{i}
    + \alpha \pi_{\varphi_{2},\psi_{1}} \iota_{\Pi}\sigma_{i} + \conj{\beta} \pi_{\varphi_{1},\psi_{2}} \iota_{\Pi}\sigma_{i}
    + \alpha \conj{\beta} \pi_{\varphi_{2},\psi_{2}} \iota_{\Pi}\sigma_{i}
  \end{align*}
   for all $\varphi_{1},\varphi_{2} \in \mathcal{X}$, $\psi_{1},\psi_{2} \in \mathcal{Y}$ and $\alpha, \beta \in \C$.
   By continuity of the projections $\pi_{\zeta,\eta}$ we conclude that the last identity also holds if we replace $\iota_{\Pi}\sigma_{i}$ by its limit $f$, i.e.,
   $f(z)$ is sesquilinear for all $z\in U$.
  Therefore, \Cref{th:very-weak-holomorphy-for-Lb-and-Ses-implies-holomorphy} implies that $\iota_{\Pi}^{-1} f \in \Hol(U,\Ses(\mathcal{X},\mathcal{Y}))$ exists. Hence, $\iota_{\Pi}\big(\Hol(U,\Ses(\mathcal{X},\mathcal{Y}))\big)$ is closed in $\prod_{(\varphi,\psi)} \Hol(U,\C)$.

  It is straightforward to show
  \begin{equation*}
    \iota_{\Pi}(K) = \prod_{(\varphi,\psi) \in \mathcal{X} \times \mathcal{Y}} \Hol(U,\cl{\ball}_{\norm{\varphi}_{\mathcal{X}} \norm{\psi}_{\mathcal{Y}}}(0)) \cap \iota_{\Pi}\big(\Hol(U,\Ses(\mathcal{X},\mathcal{Y}))\big).
  \end{equation*}

  Note that by \Cref{th:Hol-U-ball-compact} $\Hol(U,\cl{\ball}_{\norm{\varphi}\norm{\psi}}(0))$ is compact for every $\varphi \in \mathcal{X}$ and $\psi \in \mathcal{Y}$. By Tychonoff's theorem $\prod_{(\varphi,\psi)} \Hol(U,\cl{\ball}_{\norm{\varphi}\norm{\psi}}(0))$ is compact and therefore $\iota_{\Pi}(K)$ is compact as the intersection of a compact and a closed set. This finally implies the compactness of $K$.

  In the separable case, let $X\subseteq\mathcal{X}$ and $Y\subseteq\mathcal{Y}$ be countable and dense. Then, we replace
  \begin{equation*}
    \prod_{(\varphi,\psi) \in \mathcal{X} \times \mathcal{Y}} \Hol(U,\C)
  \quad\text{ with }\quad
    \prod_{(\varphi,\psi) \in X \times Y} \Hol(U,\C)
  \end{equation*}
  in the definition of $\iota_{\Pi}$, and we only consider $\pi_{\zeta,\eta}$ and $\tilde{\Lambda}_{\varphi,\psi}$ with $\zeta\in X$ and $\eta\in Y$ in \Cref{fig:initial-topology-Hol-Ses-associative}.
  This gives rise to a new topology $\mathcal{T}_{\aleph_{0}}$ on $\Hol(U,\Ses(\mathcal{X},\mathcal{Y}))$.
  Since two
  bounded sesquilinear forms that coincide on $X\times Y$ also coincide on $\mathcal{X}\times\mathcal{Y}$, we still have that $\iota_{\Pi}$ is a homeomorphism onto its range w.r.t.\ $\mathcal{T}_{\aleph_{0}}$.
  Thus, $\prod_{(\varphi,\psi)}\Hol(U,\C)$ being metrisable as the countable product of metrisable spaces (cf.\ \Cref{re:compact-convergence-topology-properties}) implies that
  $(\Hol(U,\Ses(\mathcal{X},\mathcal{Y})),\mathcal{T}_{\aleph_{0}})$ is metrisable. The identity operator
  \begin{equation*}
    \operatorname{id}\colon (K,\mathcal{T}_{\tilde{\Lambda}})\to
    (K,\mathcal{T}_{\aleph_{0}})
  \end{equation*}
  clearly is continuous and bijective. Moreover, we have just shown that its domain is compact and that
  its codomain is metrisable and hence Hausdorff. Therefore, $\operatorname{id}$ is a
  homeomorphism implying that $(K,\mathcal{T}_{\tilde{\Lambda}})$ is metrisable.
\end{proof}

\Cref{th:G-from-operator-to-sesqui-is-isomorphism} now immediately yields:

\begin{corollary}\label{th:alaoglu-for-holomorphic-linear-operators}
  The set $R \coloneqq\dset{f \in \Hol(U,\Lb(\mathcal{X},\mathcal{Y}))}{\forall z\in U :\norm{f(z)} \leq 1}$ is a $\mathcal{T}_{\Lambda}$-compact subset of $\Hol(U,\Lb(\mathcal{X},\mathcal{Y}))$.
  If both $\mathcal{X}$ and $\mathcal{Y}$ are separable, then $(R,\mathcal{T}_{\Lambda})$ is metrisable.
\end{corollary}

\section{The Parameterised Nonlocal $\mathrm{H}$-Topology}\label{sec:holwop}%

We now come to the main part of this paper, the introduction and discussion of the \hyperref[def:main-topology]{parameterised nonlocal $\mathrm{H}$-topology} or parametrised Schur topology

From now on, we regard a Hilbert space $\mathcal{H}$ that can be orthogonally decomposed into two closed subspaces
\begin{equation*}
  \mathcal{H} = \mathcal{H}_{0} \oplus \mathcal{H}_{1},
\end{equation*}
and an open subset $U$ of $\C$. Furthermore, let $\alpha \in (0,\infty)^{2\times 2}$ be the matrix
\begin{equation*}
  \alpha =
  \begin{pmatrix}
    \alpha_{00} & \alpha_{01} \\
    \alpha_{10} & \alpha_{11}
  \end{pmatrix}.
\end{equation*}
We regard the following space of holomorphic functions
\begin{equation}
  \label{eq:parameterised-M}
  \mathfrak{M}(U) \coloneqq
  \dset{M \in \Hol(U,\Lb(\mathcal{H}))}{\forall z\in U : M(z) \in \mathcal{M}(\mathcal{H}_{0},\mathcal{H}_{1})}.
\end{equation}
Similarly to~\cite[Sec.~5]{Wa22}, we will introduce the topology on $\mathfrak{M}(U)$ as an initial topology w.r.t.\ the ``projections'' onto the components of the
Schur complement and expect to obtain analogous properties. Instead of the weak operator topology we will use the locally uniform weak operator topology, see~\Cref{def:locally-uniform-weak-operator-topology}.
To wit, we regard the following ``projections''
\begin{subequations}\label{eq:schur-projections}
  \begin{align}
    \Lambda_{00}&\colon\left\{
                  \begin{array}{rcl}
                    \mathfrak{M}(U)& \to & \Hol(U,\Lb(\mathcal{H}_{0})) \\
                    M &\mapsto & M_{00}(\cdot)^{-1}
                  \end{array}
                  \right.
    \\
    \Lambda_{01}&\colon\left\{
                  \begin{array}{rcl}
                    \mathfrak{M}(U)& \to & \Hol(U,\Lb(\mathcal{H}_{1},\mathcal{H}_{0})) \\
                    M &\mapsto & M_{00}(\cdot)^{-1}M_{01}
                  \end{array}
                  \right.
    \\
    \Lambda_{10}&\colon\left\{
                  \begin{array}{rcl}
                    \mathfrak{M}(U)& \to & \Hol(U,\Lb(\mathcal{H}_{0},\mathcal{H}_{1})) \\
                    M &\mapsto & M_{10}M_{00}(\cdot)^{-1}
                  \end{array}
                  \right.
    \\
    \Lambda_{11}&\colon\left\{
                  \begin{array}{rcl}
                    \mathfrak{M}(U)& \to & \Hol(U,\Lb(\mathcal{H}_{1})) \\
                    M &\mapsto & M_{11} - M_{10}M_{00}(\cdot)^{-1}M_{01}
                  \end{array}
                  \right.
  \end{align}
\end{subequations}
Note that these mappings are well-defined by the definition of $\mathfrak{M}(U)$ and \Cref{le:components-holomorphic,le:inverse-holomorphic,le:product-holomorphic}.

\begin{definition}\label{def:main-topology}
  Let $\mathfrak{M}(U)$ be the set defined in~\eqref{eq:parameterised-M} and $\Lambda_{00}$, $\Lambda_{01}$, $\Lambda_{10}$, $\Lambda_{11}$ the mappings defined in~\eqref{eq:schur-projections}. Then we define the \emph{parameterised nonlocal $\mathrm{H}$-topology} or \emph{parameterised Schur topology}, $\tau_{\Hol}(\mathcal{H}_0,\mathcal{H}_1)$, as the initial topology on $\mathfrak{M}(U)$ w.r.t.\ the mappings $\Lambda_{00}$, $\Lambda_{01}$, $\Lambda_{10}$, $\Lambda_{11}$, where the codomains are each endowed with the corresponding \hyperref[def:locally-uniform-weak-operator-topology]{locally uniform weak operator topology $\mathcal{T}_{\Lambda}$} (\Cref{def:locally-uniform-weak-operator-topology}).
\end{definition}

\begin{remark}\label{re:paramet-implies-pointw-nonlocal-h-conv}
  Comparing \Cref{def:nHc} and \Cref{def:main-topology}, we can deduce
  that if a net $(M_{i})_{i\in I}$ in $\mathfrak{M}(U)$ \hyperref[def:main-topology]{$\tau_{\Hol}(\mathcal{H}_0,\mathcal{H}_1)$}-converges to $M\in\mathfrak{M}(U)$, then for each $z\in U$ the net $(M_{i}(z))_{i\in I}$ in $\mathcal{M}(\mathcal{H}_{0},\mathcal{H}_{1})$ converges to $M(z)\in \mathcal{M}(\mathcal{H}_{0},\mathcal{H}_{1})$ w.r.t. \hyperref[def:nHc]{$\tau(\mathcal{H}_0,\mathcal{H}_1)$}.
\end{remark}

\begin{remark}
  Apparently, we cannot expect a statement similar to \Cref{re:loc-unif-wot-is-tvs-and-hausdorff} to hold. Neither is $\mathfrak{M}(U)$ a vector space, nor are
  the mappings $\Lambda_{00}$, $\Lambda_{01}$, $\Lambda_{10}$, $\Lambda_{11}$ linear.
  Just considering $\Lambda_{00}$ and the fact that in general $1/z+1/w\neq 1/(z+w)$
  holds for $z,w\in\C$, we see that addition is not even continuous when staying within $\mathfrak{M}(U)$. Scalar multiplication however is continuous as a mapping
  from $\C \setminus \set{0}\times\mathfrak{M}(U)$ to $\mathfrak{M}(U)$ as one can show using nets and the definition of $\mathcal{T}_{\Lambda}$. Moreover, $\Lambda_{00}$, $\Lambda_{01}$, $\Lambda_{10}$, $\Lambda_{11}$ separate points, i.e., $\mathfrak{M}(U)$ is Hausdorff.
\end{remark}

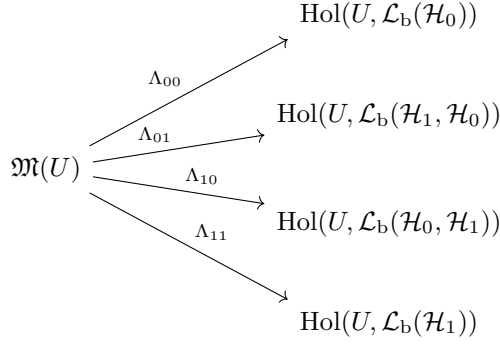
\begin{figure}
  \centering
  \begin{tikzcd}[row sep=0.7ex,column sep=15ex]
    & \Hol(U,\Lb(\mathcal{H}_{0}))  \\
    \phantom{\Hol}\\
    & \Hol(U,\Lb(\mathcal{H}_{1},\mathcal{H}_{0}))  \\
    \mathfrak{M}(U)
    \ar[end anchor=south west]{uuur}{\Lambda_{00}}
    \ar{ur}{\Lambda_{01}}
    \ar{dr}{\Lambda_{10}}
    \ar[end anchor=north west]{dddr}{\Lambda_{11}} \\
    & \Hol(U,\Lb(\mathcal{H}_{0},\mathcal{H}_1))  \\
    \phantom{\Hol}\\
    & \Hol(U,\Lb(\mathcal{H}_1))
  \end{tikzcd}
  \caption{Initial topology on $\mathfrak{M}(U)$}
\end{figure}

We now regard the space
\begin{equation*}
  \mathfrak{M}(U,\alpha) \coloneqq
  \dset{f \in \mathfrak{M}(U)}{\forall z\in U : f(z) \in \mathcal{M}(\alpha)}
\end{equation*}
equipped with the trace topology of $(\mathfrak{M}(U), \hyperref[def:main-topology]{\tau_{\Hol}(\mathcal{H}_0,\mathcal{H}_1)})$ and the spaces

\begin{align*}
  \mathcal{A}(U,\alpha_{00},\alpha_{11}) &\coloneqq \dset*{A \in \Hol(U,\Lb(\mathcal{H}_{0}))}{\forall z\in U : \Re A(z)^{-1} \geq \alpha_{00}, \Re A(z) \geq \frac{1}{\alpha_{11}}}, \\
  \mathcal{B}(U,\alpha_{01}) &\coloneqq \dset*{B \in \Hol(U,\Lb(\mathcal{H}_{0},\mathcal{H}_{1}))}{\forall z\in U : \norm{B(z)} \leq \alpha_{01}\vphantom{\frac{1}{2}}}, \\
  \mathcal{C}(U,\alpha_{10}) &\coloneqq \dset*{C \in \Hol(U,\Lb(\mathcal{H}_{1},\mathcal{H}_{0}))}{\forall z\in U : \norm{C(z)} \leq \alpha_{10}\vphantom{\frac{1}{2}}}, \\
  \mathcal{D}(U,\alpha_{00},\alpha_{11}) &\coloneqq \dset*{D \in \Hol(U,\Lb(\mathcal{H}_{1}))}{\forall z\in U : \Re D(z) \geq \alpha_{00}, \Re D(z)^{-1} \geq \frac{1}{\alpha_{11}}}
\end{align*}
equipped with the traces of the respective
\hyperref[def:locally-uniform-weak-operator-topology]{locally uniform weak operator topology}.


\begin{remark}
  Let $\Lambda_{00}$, $\Lambda_{01}$, $\Lambda_{10}$ and $\Lambda_{11}$ be the mappings from \eqref{eq:schur-projections}. Then their following restrictions to $\mathfrak{M}(U,\alpha)$ are well-defined and continuous:
  \begin{align*}
    \Lambda_{00}&\colon\left\{
                  \begin{array}{rcl}
                    \mathfrak{M}(U,\alpha)& \to & \mathcal{A}(U,\alpha_{00},\alpha_{11}) \\
                    M &\mapsto & M_{00}(\cdot)^{-1}
                  \end{array}
                  \right.
    \\
    \Lambda_{01}&\colon\left\{
                  \begin{array}{rcl}
                    \mathfrak{M}(U,\alpha)& \to & \mathcal{B}(U,\alpha_{01}) \\
                    M &\mapsto & M_{00}(\cdot)^{-1}M_{01}
                  \end{array}
                  \right.
    \\
    \Lambda_{10}&\colon\left\{
                  \begin{array}{rcl}
                    \mathfrak{M}(U,\alpha)& \to & \mathcal{C}(U,\alpha_{10}) \\
                    M &\mapsto & M_{10}M_{00}(\cdot)^{-1}
                  \end{array}
                  \right.
    \\
    \Lambda_{11}&\colon\left\{
                  \begin{array}{rcl}
                    \mathfrak{M}(U,\alpha)& \to & \mathcal{D}(U,\alpha_{00},\alpha_{11}) \\
                    M &\mapsto & M_{11} - M_{10}M_{00}(\cdot)^{-1}M_{01}
                  \end{array}
                  \right.
  \end{align*}
  In fact, they even induce the topology on $\mathfrak{M}(U,\alpha)$ as their initial
  topology (the diagram in \Cref{fig:different-intial-topos-on-M-alpha} is commutative
  and initial topologies are transitive).
\end{remark}

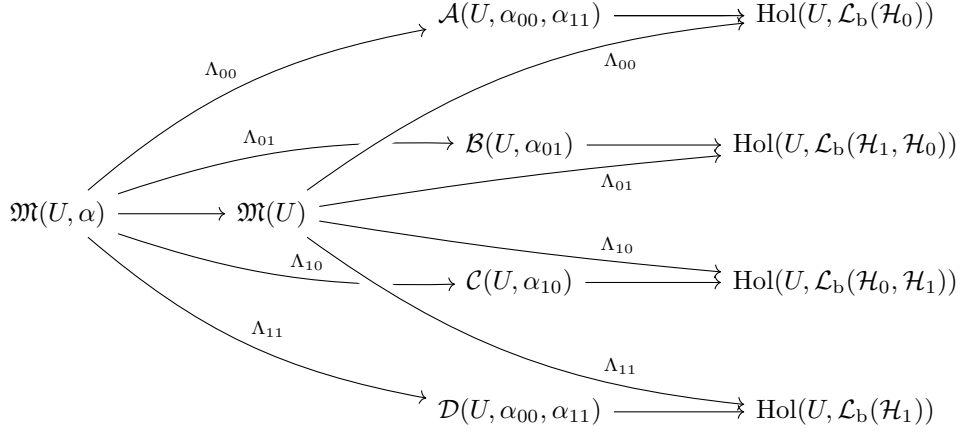
\begin{figure}
  \centering
  \begin{tikzcd}[row sep=2ex,column sep=9.4ex]
    &&\mathcal{A}(U,\alpha_{00},\alpha_{11}) \ar{r}& \Hol(U,\Lb(\mathcal{H}_{0}))  \\
    \phantom{\Hol}\\
    &&\mathcal{B}(U,\alpha_{01})\ar{r}& \Hol(U,\Lb(\mathcal{H}_{1},\mathcal{H}_{0}))  \\
    \mathfrak{M}(U,\alpha) \ar{r}
    \ar[bend left=15]{uuurr}{\Lambda_{00}}
    \ar[bend left=10]{urr}{\Lambda_{01}}
    \ar[bend right=10]{drr}{\Lambda_{10}}
    \ar[bend right=15]{dddrr}{\Lambda_{11}}
    &
    \mathfrak{M}(U)
    \ar[bend left=15,crossing over]{uuurr}[swap,pos=0.7]{\Lambda_{00}}
    \ar[bend left=2]{urr}[swap,pos=0.7]{\Lambda_{01}}
    \ar[bend right=2]{drr}[pos=0.7]{\Lambda_{10}}
    \ar[bend right=15,crossing over]{dddrr}[pos=0.7]{\Lambda_{11}}
    \\
    &&\mathcal{C}(U,\alpha_{10})\ar{r}& \Hol(U,\Lb(\mathcal{H}_{0},\mathcal{H}_1))  \\
    \phantom{\Hol}\\
    &&\mathcal{D}(U,\alpha_{00},\alpha_{11})\ar{r}& \Hol(U,\Lb(\mathcal{H}_1))
    \ar[from=1-3, to=1-4]
  \end{tikzcd}
  \caption{Two equivalent ways of defining the initial topology on $\mathfrak{M}(U,\alpha)$}
  \label{fig:different-intial-topos-on-M-alpha}
\end{figure}

\begin{lemma}\label{le:compactness-of-schur-components}
  $\mathcal{A}(U,\alpha_{00},\alpha_{11})$, $\mathcal{B}(U,\alpha_{01})$, $\mathcal{C}(U,\alpha_{10})$ and $\mathcal{D}(U,\alpha_{00},\alpha_{11})$ are compact. 
  If $\mathcal{H}$ is separable, then these sets are metrisable.
\end{lemma}

\begin{proof}
  Note that
  \begin{equation*}
    \mathcal{A}(U,\alpha_{00},\alpha_{11}) \subseteq \dset{f \in \Hol(U,\Lb(\mathcal{H}_{0}))}{\forall z\in U :\norm{f(z)} \leq C}
  \end{equation*}
  for a suitable $C \geq 0$ (see \Cref{le:realpart-bounded-from-below-implies-invertible}).
  Hence by \Cref{th:alaoglu-for-holomorphic-linear-operators}, the set $\mathcal{A}(U,\alpha_{00},\alpha_{11})$ is contained in a compact set and is therefore relatively compact.

  For compactness, we will show that $\mathcal{A}(U,\alpha_{00},\alpha_{11})$ even is closed: Let $(A_{i})_{i\in I}$ be a net in $\mathcal{A}(U,\alpha_{00},\alpha_{11})$ converging
  to $A\in \Hol(U,\Lb(\mathcal{H}_{0}))$. Then for each $z\in U$, $i\in I$ and $\varphi \in\mathcal{H}_{0}$,
  \begin{align*}
    \frac{1}{\alpha_{11}} \norm{\varphi}_{\mathcal{H}_{0}}^{2} \leq \Re \scprod{A_{i}(z) \varphi}{\varphi}_{\mathcal{H}_{0}}
  \end{align*}
  holds and taking the limit in $i\in I$ implies
  \begin{equation*}
     \frac{1}{\alpha_{11}}\leq\Re A(z).
  \end{equation*}
  These last two inequalities also show that $A_{i}(z)^{-1}\in \Lb(\mathcal{H}_{0}),i\in I$
  and $A(z)^{-1}\in \Lb(\mathcal{H}_{0})$ exist for every $z\in U$ (\Cref{le:realpart-bounded-from-below-implies-invertible}).
  Next, we have
  \begin{subequations}\label{eq:bound-on-inverse-equivalent-bound-on-operator}
  \begin{equation}
    \alpha_{00}\norm{\varphi}_{\mathcal{H}_{0}}^{2} \leq\Re\scprod{A_{i}(z)^{-1}\varphi}{\varphi}_{\mathcal{H}_{0}}
    \end{equation}
  for each $z\in U$, $i\in I$ and $\varphi\in \mathcal{H}_{0}$. Using the substitution
  $\psi = A_{i}(z)^{-1}\varphi$, we get
  \begin{equation}
    \alpha_{00}\norm{A_{i}(z)\psi}_{\mathcal{H}_{0}}^{2} \leq\Re\scprod{\psi}{A_{i}(z)\psi}_{\mathcal{H}_{0}}
    =\Re\scprod{A_{i}(z)\psi}{\psi}_{\mathcal{H}_{0}}
  \end{equation}
  \end{subequations}
  for each $z\in U$, $i\in I$ and $\psi\in \mathcal{H}_{0}$. The Cauchy--Schwarz inequality
  yields
  \begin{align}
    \alpha_{00}\frac{\abs{\scprod{A_{i}(z)\psi}{A(z)\psi}_{\mathcal{H}_{0}}}^{2}}{\norm{A(z)\psi}^{2}_{\mathcal{H}_{0}}}
    \leq \alpha_{00}\norm{A_{i}(z)\psi}_{\mathcal{H}_{0}}^{2}
    \leq\Re\scprod{A_{i}(z)\psi}{\psi}_{\mathcal{H}_{0}}\label{eq:cauchy-schwarz-instead-lim-inf}
  \end{align}
  in case $A(z)\psi\neq 0$. Taking the limits in the two scalar products, we obtain
  \begin{equation*}
    \alpha_{00}\norm{A(z)\psi}_{\mathcal{H}_{0}}^{2}
    \leq\Re\scprod{A(z)\psi}{\psi}_{\mathcal{H}_{0}}
  \end{equation*}
  for each $z\in U$ and $\psi\in \mathcal{H}_{0}$ (the case $A(z)\psi=0$ is trivial).
  The substitution $\psi = A(z)^{-1}\varphi$ then implies $\Re A(z)^{-1} \geq \alpha_{00}$.

  The same proof also shows that $\mathcal{D}(U,\alpha_{00},\alpha_{11})$ is compact.
  The sets $\mathcal{B}(U,\alpha_{01})$ and $\mathcal{C}(U,\alpha_{10})$ are already of a
  form that allows us to directly employ \Cref{th:alaoglu-for-holomorphic-linear-operators}.

  If $\mathcal{H}$ is separable, both $\mathcal{H}_{0}$ and $\mathcal{H}_{1}$ are separable
  too. Hence, \Cref{th:alaoglu-for-holomorphic-linear-operators} yields
  the desired metrisability.
\end{proof}

\begin{theorem}\label{th:main-theorem}
  $\mathfrak{M}(U,\alpha)$ equipped with the trace topology of $\mathfrak{M}(U)$ is compact.\newline
  If $\mathcal{H}$ is separable, then $\mathfrak{M}(U,\alpha)$ is metrisable and thus sequentially compact.
\end{theorem}

\begin{figure}[!h]
  \centering
  \begin{tikzcd}[row sep=0.7ex,column sep=15ex]
    &&& \mathcal{A}(U,\alpha_{00},\alpha_{11})  \\
    \phantom{\Hol}\\
    &&& \mathcal{B}(U,\alpha_{01})  \\
    \mathfrak{M}(U,\alpha)\ar{rr}{(\Lambda_{00},\Lambda_{01},\Lambda_{10},\Lambda_{11})}\ar[bend right=20]{rrrddd}{\Lambda_{11}}\ar[bend left=20]{rrruuu}{\Lambda_{00}}
    \ar[bend left=15]{rrru}{\Lambda_{01}}\ar[bend right=15]{rrrd}{\Lambda_{10}}
    &&S
    \ar[end anchor=south west,crossing over]{uuur}
     \ar[pos=0.8]{ur}
    \ar{dr}
    \ar[end anchor=north west,crossing over]{dddr}\\
    &&& \mathcal{C}(U,\alpha_{10})  \\
    \phantom{\Hol}\\
    &&& \mathcal{D}(U,\alpha_{00},\alpha_{11})
  \end{tikzcd}
  \caption{Topology on $\mathfrak{M}(U,\alpha)$ as a product topology}
  \label{fig:nonlocal-h-topology-is-product-topology}
\end{figure}
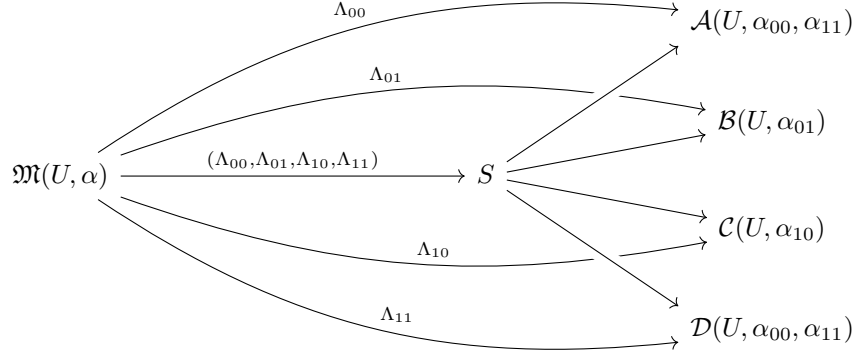

\begin{proof}
If we endow
  \begin{equation*}
    S \coloneqq
    \mathcal{A}(U,\alpha_{00},\alpha_{11})\times\mathcal{B}(U,\alpha_{01})\times \mathcal{C}(U,\alpha_{10})\times\mathcal{D}(U,\alpha_{00},\alpha_{11})
  \end{equation*}
  with the product topology, $S$ is compact (metrisable) as the product
  of finitely many compact (metrisable) spaces (see \Cref{le:compactness-of-schur-components}).
  Moreover, the mapping
  \begin{equation*}
    (\Lambda_{00},\Lambda_{01},\Lambda_{10},\Lambda_{11})\colon \mathfrak{M}(U,\alpha)\to S
  \end{equation*}
  is continuous.
  For $(A,B,C,D)\in S$ we define (note that $A(z)$ and
  $D(z)$ are bounded and invertible for $z\in U$ by \Cref{le:realpart-bounded-from-below-implies-invertible} and note
  \Cref{le:components-holomorphic,le:inverse-holomorphic,le:product-holomorphic}) the block operator
    \begin{equation*}
    z\mapsto M(z) \coloneqq
    \begin{pmatrix}
      A(z)^{-1} & A(z)^{-1}B(z) \\
      C(z) A(z)^{-1} & D(z) +C(z) A(z)^{-1} B(z)
    \end{pmatrix}\in \Hol(U,\Lb(\mathcal{H}))\text{.}
  \end{equation*}
  Its inverse operator is given by
  \begin{equation*}
    z\mapsto M(z)^{-1} =
    \begin{pmatrix}
      A(z) + B(z)D(z)^{-1}C(z) & -B(z)D(z)^{-1} \\
      -D(z)^{-1}C(z) & D(z)^{-1}
    \end{pmatrix}\in \Hol(U,\Lb(\mathcal{H}))\text{.}
  \end{equation*}
  In other words, we obtain $M\in\mathfrak{M}(U)$. Moreover,
  $\Lambda_{00}(M)=A$, $\Lambda_{01}(M)=B$, $\Lambda_{10}(M)=C$ and
  $\Lambda_{11}(M)=D$, i.e., $M\in \mathfrak{M}(U,\alpha)$ is a
  pre-image of $(A,B,C,D)$ under $(\Lambda_{00},\Lambda_{01},\Lambda_{10},\Lambda_{11})$.
  Clearly, it is the only one, which implies that 
  \begin{equation*}(\Lambda_{00},\Lambda_{01},\Lambda_{10},\Lambda_{11})\colon \mathfrak{M}(U,\alpha)\to S\end{equation*}
  is a continuous bijection. Since the diagram in \Cref{fig:nonlocal-h-topology-is-product-topology} is commutative and initial topologies are transitive,
  $(\Lambda_{00},\Lambda_{01},\Lambda_{10},\Lambda_{11})$ even is a homeomorphism, which finishes the proof.
\end{proof}

\begin{lemma}\label{le:paramet-equiv-pointw-nonlocal-h-conv}
  Let $(M_{i})_{i\in I}$ be a net in $\mathfrak{M}(U,\alpha)$ and $M\colon U\to\mathcal{M}(\mathcal{H}_{0},\mathcal{H}_{1})$.
  Then, $M\in \mathfrak{M}(U)$ and $(M_{i})_{i\in I}$ converges to $M$ w.r.t.\ \hyperref[def:main-topology]{$\tau_{\Hol}(\mathcal{H}_0,\mathcal{H}_1)$} if and only
  if $(M_{i}(z))_{i\in I}$ converges to $M(z)$ w.r.t.\ \hyperref[def:nHc]{$\tau(\mathcal{H}_0,\mathcal{H}_1)$}  for every $z\in U$. In either case, we have $M\in \mathfrak{M}(U,\alpha)$.
\end{lemma}

\begin{proof}
  As discussed in \Cref{re:paramet-implies-pointw-nonlocal-h-conv}, we know that
  parameterised implies pointwise convergence to the same limit.

  Conversely, assume that
  $(M_{i}(z))_{i\in I}$ converges to $M(z)$ w.r.t.\ \hyperref[def:nHc]{$\tau(\mathcal{H}_0,\mathcal{H}_1)$} for every
  $z\in U$ and consider any subnet of $(M_{i})_{i\in I}$. By virtue of \Cref{th:main-theorem}, this subnet
  has a further subnet converging to some $N\in \mathfrak{M}(U,\alpha)$ w.r.t.\ the \hyperref[def:main-topology]{$\tau_{\Hol}(\mathcal{H}_0,\mathcal{H}_1$)}. Since we also have pointwise convergence to $M$, \Cref{re:paramet-implies-pointw-nonlocal-h-conv} implies $N=M\in\mathfrak{M}(U,\alpha)$. So, every subnet has a further subnet converging to $M$ which finishes the proof.
\end{proof}

Combining \Cref{le:paramet-equiv-pointw-nonlocal-h-conv} and \Cref{le:sot-conv-implies-nonlocal-h-pointw}, we immediately obtain:

\begin{corollary}\label{cor:sotnlH}
  Let $(M_{n})_{n\in\N}$ be a sequence in $\mathfrak{M}(U,\alpha)$
  and $M\colon U\to \Lb(\mathcal{H})$
  such that $M_{n}(z)$ converges to $M(z)$ in the strong operator topology for every $z \in U$. Then, $M \in \mathfrak{M}(U,\alpha)$ and $(M_{n})_{n\in\N}$ converges to $M$ w.r.t.\
  \hyperref[def:main-topology]{$\tau_{\Hol}(\mathcal{H}_0,\mathcal{H}_1)$}.
\end{corollary}

We stress that the statement of \Cref{cor:sotnlH} is independent of the decomposition considered for $\mathcal{H}$. In the next section, we establish the announced continuity result for solution operators for abstract time-dependent partial differential equations, that is, for evolutionary equations.

\section{Applications to Evolutionary Equations}\label{sec:AEEs}

Finally, we establish the connection between evolutionary equations and the \hyperref[def:main-topology]{parameterised nonlocal $\mathrm{H}$-topology}.

The following lemma is based on ideas obtained from~\cite{BV14,BEW2023}.

\begin{lemma}\label{le:wot-of-solutions-conserves-inequalities}
  Let $\mathcal{H}$ be a separable Hilbert space and $(T_{n})_{n\in\N}$ in $ \Lb(\mathcal{H})$ with
\begin{equation*}
  \Re T_{n}\geq c >0 \quad\text{ and }\quad
  \Re T_{n}^{-1}\geq d > 0
\end{equation*}
  for $n\in\N$. Moreover, assume that $A\colon\dom (A)\subseteq\mathcal{H}\to \mathcal{H}$ is a
  skew-selfadjoint operator and that $T\in\Lb(\mathcal{H})$ such that $0\in\rho (T+A)$. If $(T_{n}+A)^{-1}$ converges to
  $(T+A)^{-1}$ in the weak operator topology, then we obtain
  \begin{equation*}
  \Re T\geq c \quad\text{ and }\quad
  \Re T^{-1}\geq d\text{.}
\end{equation*}
\end{lemma}
\begin{proof}
  $T_{n}+A$ is closed with adjoint $T^{\ast}_{n}-A$ and
\begin{equation}
  \Re\scprod{(T_{n}+A)\varphi}{\varphi}_{\mathcal{H}} =\Re\scprod{T_{n}\varphi}{\varphi}_{\mathcal{H}}\geq c\norm{\varphi}^{2}_{\mathcal{H}} \label{eq:real-part-skew-is-zero}
\end{equation}
  for all $\varphi\in\dom (A)$ and $n\in\N$.
  By virtue of \Cref{le:ubd-realpart-bounded-from-below-implies-invertible}, we infer $(T_{n}+A)^{-1}\in \Lb(\mathcal{H})$ with $\norm{(T_{n}+A)^{-1}}\leq 1/c$.
  \Cref{le:realpart-bounded-from-below-implies-invertible} yields $\norm{T_{n}}\leq 1/d$. Hence, we also get
\begin{equation}\label{eq:resolvent-bounded-in-graph-space}
  \norm{A(T_{n}+A)^{-1}}=\norm{I-T_{n}(T_{n}+A)^{-1}}\leq 1+\frac{1}{cd}
\end{equation}
for $n\in\N$. Thus, for any subsequence of $((T_{n}+A)^{-1})_{n\in\N}$, sequential compactness of operator norm balls in the weak operator topology gives us a
further subsequence that converges in  the weak operator topology on $\Lb(\mathcal{H},\dom(A))$, where the Hilbert space $\dom(A)$ is endowed with the graph inner product.
Since the weak operator limits in $\Lb(\mathcal{H},\dom(A))$ and $\Lb(\mathcal{H})$ have to coincide (
$\dom(A)$ as a Hilbert space is continuously embedded in $\mathcal{H}$), we conclude that every subsequence of $((T_{n}+A)^{-1})_{n\in\N}$ has
a further subsequence converging to $(T+A)^{-1}$ in  the weak operator topology on $\Lb(\mathcal{H},\dom(A))$. In other words, $(T_{n}+A)^{-1}$ converges to
$(T+A)^{-1}$ in  the weak operator topology on $\Lb(\mathcal{H},\dom(A))$.

We have $T_{n}(T_{n}+A)^{-1}=I-A(T_{n}+A)^{-1}$ for $n\in\N$. Therefore, $T_{n}(T_{n}+A)^{-1}$ converges to $I-A(T+A)^{-1}=T(T+A)^{-1}$ in  the weak operator topology on $\Lb(\mathcal{H})$. Furthermore, the skew-selfadjointness of $A$ implies
\begin{equation*}
  \Re\scprod{A(T_{n}+A)^{-1}\varphi}{(T_{n}+A)^{-1}\varphi}_{\mathcal{H}}= 0\text{ and }\Re\scprod{A(T+A)^{-1}\varphi}{(T+A)^{-1}\varphi}_{\mathcal{H}}= 0
  \end{equation*}
  for all $\varphi\in\mathcal{H}$ and $n\in\N$. Thus, from\begin{equation*}
  T_{n}(T_{n}+A)^{-1}+A(T_{n}+A)^{-1}=I=T(T+A)^{-1}+A(T+A)^{-1}
\end{equation*}
for all $n\in\N$, it follows
\begin{equation}\label{eq:conv-real-part-Tn-TnA}
\begin{aligned}
  \lim_{n\to\infty}\Re\scprod{T_{n}(T_{n}+A)^{-1}\varphi}{(T_{n}+A)^{-1}\varphi}_{\mathcal{H}}
  &=\lim_{n\to\infty}\Re \scprod{\varphi}{(T_{n}+A)^{-1}\varphi}_{\mathcal{H}}\\
  &=\Re \scprod{\varphi}{(T+A)^{-1}\varphi}_{\mathcal{H}}\\
  &=\Re\scprod{T(T+A)^{-1}\varphi}{(T+A)^{-1}\varphi}_{\mathcal{H}}
\end{aligned}
\end{equation}
for all $\varphi\in\mathcal{H}$. Reusing the methods employed in~\eqref{eq:cauchy-schwarz-instead-lim-inf} and~\eqref{eq:bound-on-inverse-equivalent-bound-on-operator}, we obtain
\begin{align*}
  c\norm{(T+A)^{-1}\varphi}^{2}_{\mathcal{H}}&\leq\lim_{n\to\infty} \Re\scprod{T_{n}(T_{n}+A)^{-1}\varphi}{(T_{n}+A)^{-1}\varphi}_{\mathcal{H}}\\
  &=\Re\scprod{T(T+A)^{-1}\varphi}{(T+A)^{-1}\varphi}_{\mathcal{H}}
\end{align*}
as well as
\begin{equation}\label{eq:inverse-real-part-limit-bound}
\begin{aligned}
  d\norm{T(T+A)^{-1}\varphi}^{2}_{\mathcal{H}}&\leq\lim_{n\to\infty} \Re\scprod{T_{n}(T_{n}+A)^{-1}\varphi}{(T_{n}+A)^{-1}\varphi}_{\mathcal{H}}\\
  &=\Re\scprod{T(T+A)^{-1}\varphi}{(T+A)^{-1}\varphi}_{\mathcal{H}}
\end{aligned}
\end{equation}
for all $\varphi\in \mathcal{H}$. As $(T+A)^{-1}\varphi$ ranges over the dense subspace $\dom(A)$ of $\mathcal{H}$ and as both $T$ and $T^{-1}$ are bounded
on $\mathcal{H}$, we conclude $\Re T\geq c$ and
with~\eqref{eq:bound-on-inverse-equivalent-bound-on-operator} also $\Re T^{-1}\geq d$.
\end{proof}

From now on, let $\mathcal{H}$ be a separable Hilbert space,
let $A\colon \dom(A)\subseteq \mathcal{H}\to\mathcal{H}$ be skew-selfadjoint and let $\dom A \cap (\ker A)^{\perp}$ endowed with the graph scalar product of
$A$ be compactly embedded into $\mathcal{H}$.
Recall that this compact embedding implies closedness of $\ran A$ by a standard argument (see~\cite[Lemma~4.1]{ElstGordWau2019} or the FA-Toolbox in \cite{PZ20}) and hence
$(\ker A)^{\perp} =\ran A$.
Thus, we obtain the following decomposition:
\begin{equation}\label{eq:decompA}
  \mathcal{H} = \mathopen{}\underbrace{\ker A}_{\eqqcolon\mathrlap{\mathcal{H}_{0}}} \mathclose{}\oplus \underbrace{\ran A}_{\eqqcolon\mathrlap{\mathcal{H}_{1}}}\text{,}
\end{equation}
and $\dom(A)\cap \mathcal{H}_{1}$ is compactly embedded in $\mathcal{H}_{1}$.
\begin{remark}
Clearly, $A$ itself has the form
\begin{equation*}
  \begin{pmatrix}
    0&0\\
    0&\tilde{A}
  \end{pmatrix}\colon \ker A\oplus(\dom A \cap\ran A)  \subseteq\mathcal{H}_{0}\oplus \mathcal{H}_{1}\to \mathcal{H}_{0}\oplus \mathcal{H}_{1}\text{,}
\end{equation*}
where $\tilde{A}\colon (\dom A \cap\ran A)\subseteq\mathcal{H}_{1}\to \mathcal{H}_{1}$ is the restriction of $A$.
Introducing this $\tilde{A}$ gives that $\dom \tilde{A}$ as a Hilbert space is compactly embedded in $\mathcal{H}_{1}$.
One can immediately verify that
$\tilde{A}$ is still skew-selfadjoint by showing that $(\tilde{A})^{\ast}=\widetilde{(A^{\ast})}=-\tilde{A}$.
\end{remark}

\begin{lemma}\label{le:inverse-of-t-plus-a}
  Consider an operator $T\in\Lb(\mathcal{H})$ and assume $T_{00}^{-1}\in\Lb(\mathcal{H})$ as well as $\Re (T_{11}-T_{10}T_{00}^{-1}T_{01})\geq c >0$.
  Then $(T+A)^{-1}\in \Lb(\mathcal{H})$ and this inverse reads

  \begin{equation}\label{eq:inverse-of-perturbed-block-matrix}
    \begin{pmatrix}
      T_{00}^{-1}+T_{00}^{-1}T_{01}T_A^{-1}T_{10}T_{00}^{-1} & -T_{00}^{-1}T_{01}T_A^{-1}\\
      -T_{A}^{-1}T_{10}T_{00}^{-1} & T_A^{-1}
    \end{pmatrix}\text{,}
  \end{equation}
where $T_{A}\coloneqq (T_{11}-T_{10}T_{00}^{-1}T_{01}+\tilde{A})$.
Moreover, we have
\begin{align*}
  \norm{T_{A}^{-1}}\leq\frac{1}{c}
  \quad\text{and}\quad
  \norm{\tilde{A}T_{A}^{-1}}\leq 1+\frac{\norm{T_{11}-T_{10}T_{00}^{-1}T_{01}}}{c}\text{.}
\end{align*}
\end{lemma}

\begin{proof}
  Using the decomposition~\eqref{eq:decompA} and \Cref{le:components-holomorphic}, we can write
  \begin{equation*}
    T + A =
    \begin{pmatrix}
      T_{00}&T_{01}\\
      T_{10}&T_{11}+\tilde{A}
    \end{pmatrix}
    \colon \mathcal{H}_{0}\oplus(\dom A \cap\mathcal{H}_{1}) \subseteq\mathcal{H}_{0}\oplus \mathcal{H}_{1}
    \to \mathcal{H}_{0}\oplus \mathcal{H}_{1}
  \end{equation*}
  with all the components of $T$ being bounded by $\norm{T}$.
  Due to \Cref{le:ubd-realpart-bounded-from-below-implies-invertible} and the conditions imposed on $T$ and $A$,
  $T_{11}-T_{10}T_{00}^{-1}T_{01}+\tilde{A}$ is boundedly invertible on $\mathcal{H}_{1}$  with
  $\norm{(T_{11}-T_{10}T_{00}^{-1}T_{01}+\tilde{A})^{-1}}_{\mathcal{H}_{1}}\leq 1/c$
  (cf.~\eqref{eq:real-part-skew-is-zero}).
  Therefore,~\eqref{eq:inverse-of-perturbed-block-matrix} is an element of $\Lb(\mathcal{H})$. Furthermore,~\eqref{eq:inverse-of-perturbed-block-matrix}
  maps from $\mathcal{H}_{0}\oplus \mathcal{H}_{1}$ to $\mathcal{H}_{0}\oplus(\dom A \cap\mathcal{H}_{1})$. It remains to verify that
  applying~\eqref{eq:inverse-of-perturbed-block-matrix} to $T+A$ from the right yields the identity on $\mathcal{H}$, and that
  applying~\eqref{eq:inverse-of-perturbed-block-matrix} to $T+A$ from the left yields the identity on $\mathcal{H}_{0}\oplus(\dom A \cap\mathcal{H}_{1})$. These
  are two short and straightforward calculations. The remaining inequality follows
  similarly to~\eqref{eq:resolvent-bounded-in-graph-space}.
\end{proof}

The combination of the definition of the Schur topology together with the compactness assumption on $A$ leads to the following fundamental convergence statement underlying our main result on evolutionary equations.

\begin{lemma}\label{le:nonlocal-h-implies-solution-conv}
  Let $(T_{n})_{n\in\N}$ be a sequence in $\mathcal{M}(\alpha)$ converging
  to $T\in \mathcal{M}(\alpha)$ w.r.t.\ \hyperref[def:nHc]{$\tau(\mathcal{H}_0,\mathcal{H}_1)=\tau(\ker(A),\ran(A))$}. Then,
  $(T_{n}+A)^{-1}$ converges to $(T+A)^{-1}$ in the weak operator topology on $\Lb(\mathcal{H})$.
\end{lemma}

\begin{proof}
  In view of \Cref{le:realpart-bounded-from-below-implies-invertible}, we can write both
  $(T_{n}+A)^{-1}$ and $(T+A)^{-1}$ in the form of~\eqref{eq:inverse-of-perturbed-block-matrix}.

  Consider any subsequence of $(T_{n}+A)^{-1}$. We will not introduce a new index for this subsequence. For $\varphi_{0}+\varphi_{1}\in\mathcal{H}_{0}\oplus\mathcal{H}_{1}$
  and $n\in\N$, we have
  \begin{equation*}
    \norm{\varphi_{1}-T_{n,10}T^{-1}_{n,00}\varphi_{0}}_{\mathcal{H}_{1}}\leq
    \norm{\varphi_{1}}_{\mathcal{H}_{1}}+\alpha_{10}\norm{\varphi_{0}}_{\mathcal{H}_{0}}\text{.}
  \end{equation*}
  Moreover, \Cref{le:inverse-of-t-plus-a} yields
  \begin{align*}
  \norm{(T_{n,11}-T_{n,10}T_{n,00}^{-1}T_{n,01}+\tilde{A})^{-1}}&\leq\frac{1}{\alpha_{00}}\\
  \mathllap{\text{and}}\quad\norm{\tilde{A}(T_{n,11}-T_{n,10}T_{n,00}^{-1}T_{n,01}+\tilde{A})^{-1}}&\leq 1+\frac{\alpha_{11}}{\alpha_{00}}\text{.}
  \end{align*}
  Thus, denoting 
  \begin{equation*}
    (u_{n,1})_{n\in\N}\coloneqq \big((T_{n,11}-T_{n,10}T_{n,00}^{-1}T_{n,01}+\tilde{A})^{-1}(\varphi_{1}-T_{n,10}T^{-1}_{n,00}\varphi_{0})\big)_{n\in\N}
  \end{equation*}
  and using
  \begin{equation*}
    \norm{T_{n,11}-T_{n,10}T_{n,00}^{-1}T_{n,01}}\leq\alpha_{11}
  \end{equation*}
  as well as
  \begin{equation*}
\norm{T_{n,10}T_{n,00}^{-1}}\leq \alpha_{10}
\end{equation*}
  for $n\in\N$, we deduce that both $(u_{n,1})_{n\in \N}$ and
  \begin{equation*}
    (\tilde{A}u_{n,1})_{n\in\N}= (\varphi_{1}-T_{n,10}T^{-1}_{n,00}\varphi_{0}- (T_{n,11}-T_{n,10}T_{n,00}^{-1}T_{n,01})u_{n,1})_{n\in\N}
  \end{equation*}
  are bounded sequences in $\mathcal{H}_1$.
  Hence, we may choose a subsequence (not relabelled) such that $(u_{n,1})_{n\in \N}$ weakly converges to some $u_1$ in $\dom(\tilde{A})$ endowed with the graph inner product. Since the continuity of $\tilde{A}\colon \dom(\tilde{A})\to \mathcal{H}_1$ (w.r.t.\ the graph norm) implies its weak continuity, the sequence $\tilde{A}u_{n,1}$ weakly converges to $\tilde{A}u_1$.
  The compact embedding of $\dom(\tilde{A})$ into $\mathcal{H}_{1}$ yields $\mathcal{H}_1$-convergence of (a subsequence of) $(u_{n,1})_{n\in \N}$ to $u_{1}\in \mathcal{H}_{1}$.
  Next, consider $((T_{n,11}-T_{n,10}T_{n,00}^{-1}T_{n,01})u_{n,1})_{n\in\N}$. As this is
  a uniformly bounded sequence of operators converging in the weak operator topology (\hyperref[def:nHc]{$\tau(\mathcal{H}_0,\mathcal{H}_1)$}-convergence)
  applied to a convergent sequence in $\mathcal{H}_{1}$, the sequence altogether weakly converges to $(T_{11}-T_{10}T_{00}^{-1}T_{01})u_{1}$.
  All in all, we have proven
  \begin{equation*}
    \tilde{A}u_{1}=\varphi_{1}-T_{10}T^{-1}_{00}\varphi_{0}- (T_{11}-T_{10}T_{00}^{-1}T_{01})u_{1}\text{,}
  \end{equation*}
   i.e.,
\begin{equation*}
  u_{1}= (T_{11}-T_{10}T_{00}^{-1}T_{01}+\tilde{A})^{-1}(\varphi_{1}-T_{10}T^{-1}_{00}\varphi_{0})\text{.}
\end{equation*}
In other words, the $\mathcal{H}_{1}$-component of $(T_{n}+A)^{-1}(\varphi_{0}+\varphi_{1})$ converges
to the $\mathcal{H}_{1}$-component of $(T+A)^{-1}(\varphi_{0}+\varphi_{1})$.

The convergence of $(u_{n,1})_{n\in\N}$, \hyperref[def:nHc]{$\tau(\mathcal{H}_0,\mathcal{H}_1)$}-convergence and the uniform bound
\begin{equation*}
\norm{T_{n,00}^{-1}T_{n,01}}\leq \alpha_{01}
\end{equation*}
for $n\in\N$ yield weak convergence of $T_{n,00}^{-1}\varphi_{0}-T_{n,00}^{-1}T_{n,01}u_{n,1}$
to $T_{00}^{-1}\varphi_{0}-T_{00}^{-1}T_{01}u_{1}$. In other words, the $\mathcal{H}_{0}$-component of $(T_{n}+A)^{-1}(\varphi_{0}+\varphi_{1})$ weakly converges 
to the $\mathcal{H}_0$-component of $(T+A)^{-1}(\varphi_{0}+\varphi_{1})$.

To sum up, we have shown that every subsequence of $(T_{n}+A)^{-1}$
has a further subsequence that converges to $(T+A)^{-1}$ in the weak
operator topology.
\end{proof}

We are now in the position to state and prove the main result of this section.

\begin{theorem}\label{th:second-main-theorem}
  Consider $\nu_{0}> 0$ and a sequence of material laws $(M_{n})_{n\in\N}$ with
  $\C_{\Re>\nu_{0}}$ in their domain. Furthermore, assume there exist $c,d>0$ with
  \begin{equation*}
    \Re zM_{n}(z) \geq c \quad\text{and}\quad \norm{M_{n}(z)} \leq d
\end{equation*}
for all $z\in\C_{\Re>\nu_{0}}$ and all $n\in\N$. This implies $\absbd(M_{n})\leq\nu_{0}$ for all $n\in\N$.
If there exists an $M\colon \C_{\Re>\nu_{0}}\to\mathcal{M}(\ker(A),\ran(A))$ with
$\norm{M(z)} \leq d$ for $z\in\C_{\Re >\nu_{0}}$ and  $(M_{n})_{n\in\N}$ converges to $M$
pointwise in \hyperref[def:nHc]{$\tau(\ker(A),\ran(A))$}, then $M$ is a material law with
\begin{equation}\label{eq:material-law-limit-bounds}
    \Re zM(z)\geq c
\end{equation}
for all $z\in\C_{\Re>\nu_{0}}$ and $\absbd(M)=\nu_{0}$.
  Moreover, we have
  \begin{equation}\label{eq:material-law-resolvent-convergence-t}
    \cl{\partial_{t}M_{n}(\partial_{t})+A}^{-1} \to \cl{\partial_{t}M(\partial_{t})+A}^{-1}
  \end{equation}
  in the weak operator topology on $\Lb(\Lp{2,\nu}(\R,\mathcal{H}))$ for every $\nu>\nu_{0}$.
\end{theorem}

\begin{proof}
  \Cref{le:realpart-bounded-from-below-implies-invertible} yields
  \begin{equation*}
    \Re (zM_{n}(z))^{-1}\geq c\norm{zM_{n}(z)}^{-2}\geq cd^{-2}\abs{z}^{-2}
  \end{equation*}
  for all $z\in\C_{\Re>\nu_{0}}$ and all $n\in\N$.
  Fix any $\mu> \nu_{0}$. Then by easy calculations (\Cref{le:components-holomorphic,le:realpart-bounded-from-below-implies-invertible,le:dohnal}), we find an
  $\alpha\in (0,\infty)^{2\times 2}$ such that $(z\mapsto zM_{n}(z))\in\mathfrak{M}(\C^{\abs{\Im}<\mu}_{\mu >\Re>\nu_{0}},\alpha)$
  for $n\in\N$. \Cref{le:paramet-equiv-pointw-nonlocal-h-conv} yields holomorphicity of $M$ and $zM(z)\in\mathcal{M}(\alpha)$  on $\C^{\abs{\Im}<\mu}_{\mu >\Re>\nu_{0}}$.
  Since $\mu>\nu_{0}$ was arbitrary, we obtain holomorphicity of $M$ on $\C_{\Re>\nu_{0}}$.

  In particular, we have proven $z M_{n}(z)\in\mathcal{M}(\alpha)$ for all $n\in\N$ and $z M(z)\in\mathcal{M}(\alpha)$ for each $z\in \C_{\Re>\nu_{0}}$ (with the $\alpha$ only depending on $z$). Thus, \Cref{le:nonlocal-h-implies-solution-conv} yields
  \begin{equation}\label{eq:material-law-resolvent-convergence-z}
    (zM_{n}(z)+A)^{-1}\to (zM(z)+A)^{-1}
  \end{equation}
  in the weak operator topology for each $z\in \C_{\Re>\nu_{0}}$, and \Cref{le:wot-of-solutions-conserves-inequalities}
  proves~\eqref{eq:material-law-limit-bounds}. This means, \Cref{th:pt} is applicable to both $M_{n}$ for $n\in\N$ and to $M$.
  Fourier--Laplace transforming~\eqref{eq:material-law-resolvent-convergence-z},
  we get~\eqref{eq:material-law-resolvent-convergence-t}.
\end{proof}

\begin{remark}\label{re:second-main-theorem}
  It is possible to replace the uniform boundedness condition imposed on $(M_{n})_{n\in\N}$ and
  its limit $M$ in \Cref{th:second-main-theorem} with
  \begin{equation}\label{eq:alternative-boundedness-sec-main-th}
    \Re \scprod{M_{n}(z)\varphi}{\varphi}_{\mathcal{H}}\geq \frac{1}{d} \norm{M_{n}(z)\varphi}_{\mathcal{H}}^{2}
  \end{equation}
  for all $z\in\C_{\Re >\nu_{0}}$ and $n\in\N$:

  First, note that
  $\Re z M_{n}(z)\geq c$ and \Cref{le:realpart-bounded-from-below-implies-invertible}
  show that $M_{n}(z)$ is boundedly invertible for all $z\in\C_{\Re >\nu_{0}}$ and $n\in\N$.
  Hence looking at~\eqref{eq:bound-on-inverse-equivalent-bound-on-operator}, we see that~\eqref{eq:alternative-boundedness-sec-main-th} is equivalent to $\Re (M_{n}(z))^{-1}\geq 1/d$
  and with \Cref{le:realpart-bounded-from-below-implies-invertible} we even obtain $\norm{M_{n}(z)}\leq d$ for all $z\in\C_{\Re >\nu_{0}}$ and $n\in\N$. Therefore, we can apply the proof of
  \Cref{th:second-main-theorem} until we get~\eqref{eq:material-law-resolvent-convergence-z}
  and~\eqref{eq:material-law-limit-bounds}.

  In order to obtain~\eqref{eq:material-law-resolvent-convergence-t}, we need to apply \Cref{th:pt}, which means, it remains to prove the uniform boundedness of $M$. Since
  we now have a compactness condition on $A$, we can refine the argument~\eqref{eq:conv-real-part-Tn-TnA}. We have
  \begin{subequations}
    \begin{align}
        &\scprod*{zM_{n}(z)(zM_{n}(z)+A)^{-1}\varphi}{(zM_{n}(z)+A)^{-1}\varphi}_{\mathcal{H}}\label{eq:refined-arg-1} \\
        &\hspace{1.1cm}\mathclose{}+
          \scprod*{A(zM_{n}(z)+A)^{-1}\varphi}{(zM_{n}(z)+A)^{-1}\varphi}_{\mathcal{H}} \label{eq:refined-arg-2} \\
        &\hspace{3.1cm}\mathclose{}=
          \scprod*{zM(z)(zM(z)+A)^{-1}\varphi}{(zM_{n}(z)+A)^{-1}\varphi}_{\mathcal{H}} \label{eq:refined-arg-3} \\
        &\phantom{\mathopen{}=\mathclose{}}\hspace{4.2cm}\mathclose{}+
          \scprod*{A(zM(z)+A)^{-1}\varphi}{(zM_{n}(z)+A)^{-1}\varphi}_{\mathcal{H}} \label{eq:refined-arg-4}
    \end{align}
  \end{subequations}
  for $\varphi\in\mathcal{H}$. \eqref{eq:refined-arg-3} and~\eqref{eq:refined-arg-4} converge
  due to~\eqref{eq:material-law-resolvent-convergence-z}.

  For~\eqref{eq:refined-arg-2}, we
  recall that we have proven weak convergence of $A(zM_{n}(z)+A)^{-1}\varphi$ to
  $A(zM(z)+A)^{-1}\varphi$ in the first paragraph of the proof of \Cref{le:wot-of-solutions-conserves-inequalities}. Moreover, $(zM_{n}(z)+A)^{-1}\varphi$ in the
  second entry of the inner product can be replaced with its projection onto
  $\mathcal{H}_{1}=\ran A$ as it is multiplied with $A(zM_{n}(z)+A)^{-1}\varphi\in\ran A$. Obviously, this projected sequence converges weakly in
  the Hilbert space $\dom(A)\cap \mathcal{H}_{1}$.
  As a consequence, we get strong convergence to the projection of $(zM(z)+A)^{-1}\varphi$ onto $\mathcal{H}_{1}$ by the compact
  embedding of $\dom(A)\cap \mathcal{H}_{1}$ into $\mathcal{H}$. Alltogether, that
  means convergence of~\eqref{eq:refined-arg-2} to $\scprod{A(zM(z)+A)^{-1}\varphi}{(zM(z)+A)^{-1}\varphi}_{\mathcal{H}}$
  and thus~\eqref{eq:refined-arg-1} converges to
  \begin{equation*}
    \scprod{zM(z)(zM(z)+A)^{-1}\varphi}{(zM(z)+A)^{-1}\varphi}_{\mathcal{H}}\text{.}
  \end{equation*}
  Dividing by $z$ and repeating the argument~\eqref{eq:inverse-real-part-limit-bound},
  we conclude $\Re (M(z))^{-1}\geq 1/d$
  and with \Cref{le:realpart-bounded-from-below-implies-invertible} even $\norm{M(z)}\leq d$ for all $z\in\C_{\Re >\nu_{0}}$.
\end{remark}

\section{Examples}\label{sec:exs}

\subsection{On a model for cell migration}

In \cite{KPSZ20}, the authors introduce and analyse a nonlocal model for cell migration. Here, we are interested to exemplify our previous findings. Hence, we only focus on an autonomous, linear variant of the equation in \cite{KPSZ20}. However, we may allow for matrix-valued coefficients here. For this, let throughout $\Omega\subseteq \R^n$ be a bounded, weak Lipschitz domain with continuous boundary, and introduce, for $r\geq 0$ and $q\in \Lp{2}(\Omega)^n$, the linear operator $\mathcal{S}_r$ given by
\begin{equation*}
   \mathcal{S}_r q (x)\coloneqq n\int_0^1\frac{1}{|S_1|}\int_{S_1} \langle q(x+r sy),y\rangle_{\R^n}y \dx[\sigma (y)] \dx[s] \quad(x\in \Omega),
\end{equation*}
where $q$ is extended to $\R^n$ via $0$, $S_{1}$ denotes the sphere with radius $1$ and $\sigma$ its surface measure.   According to \cite{KPSZ20}, we have $\mathcal{S}_r\in \Lb(\Lp{2}(\Omega)^n)$ for all $r\geq 0$. Moreover, the operator family $(\mathcal{S}_r)_{0\leq r\leq 1}$ is a special case of an approximation of unity.
\begin{definition} We call $(\mathcal{R}_r)_{0\leq r\leq 1}$ in $\Lb(\Lp{2}(\Omega)^n)$ an \emph{approximation of unity}, if $\sup_{0\leq r\leq 1}\norm{\mathcal{R}_r} < \infty$ and $\mathcal{R}_r\to 1$ in the strong operator topology as $r\to 0$.
\end{definition}
Note that the example $(\mathcal{T}_r)_r$ treated in \cite{KPSZ20} is, too, an approximation of unity.

In the following, let $(\mathcal{R}_r)_r$ be an approximation of unity. Then consider $a_1,a_2, a_3 \in M(\alpha,\beta;\Omega)$ for some $0<\alpha<\beta$ and consider, for $0\leq r\leq 1$, the following equation
\begin{equation*}
    \partial_t c_r - \div (a_1-a_2\mathcal{R}_r a_3 )\grad c_r =f\in \Lp[\nu]{2}(\R;\Lp{2}(\Omega)),
\end{equation*}
with $f$ and $\nu>0$ fixed. Introducing $q_r\coloneqq -A_r\grad c_r$ with $A_r\coloneqq (a_1-a_2\mathcal{R}_r a_3 )$ and assuming homogeneous Neumann boundary conditions for $q_r$, we rewrite the system as an evolutionary equation. For this, we impose the standing assumption that there exists $c>0$ such that for all $0\leq r\leq 1$, we have
\begin{equation*}
   \Re (a_1-a_2\mathcal{R}_r a_3 )\geq c
\end{equation*}
in the sense of positive definiteness in $\Lb(\Lp{2}(\Omega)^n)$. This assumption is slightly weaker than the one imposed in \cite{KPSZ20}. By \Cref{le:realpart-bounded-from-below-implies-invertible}, it implies that $A_{r}$ is boundedly invertible with $\norm{A_{r}^{-1}}\leq 1/c$.
Then, we may equivalently consider
\begin{equation*}
  \left[ \partial_t \begin{pmatrix}1 & 0 \\ 0 & 0\end{pmatrix}
    + \begin{pmatrix}0 & 0 \\ 0 & A_r^{-1}\end{pmatrix}
    + \begin{pmatrix} 0& \divn \\ \grad & 0\end{pmatrix}
    \right]
    \begin{pmatrix} c_r \\ q_r \end{pmatrix}
    = \begin{pmatrix} f \\ 0 \end{pmatrix}
\end{equation*}
where $\divn\coloneqq \cl{\div|_{\Cc(\Omega)^n}}$ is the closure of $\div$ as an operator in $\Lp{2}$ on smooth compactly supported vector fields. This models homogeneous Neumann boundary conditions.

Note that
\begin{equation*}
  M_r\colon z\mapsto \begin{pmatrix}1 & 0 \\ 0 & 0\end{pmatrix} +  z^{-1}
  \begin{pmatrix}0 & 0 \\ 0 & A_r^{-1}\end{pmatrix}
\end{equation*}
defines material laws for $0\leq r\leq 1$ with $\absbd (M_{r})=0$ and the following properties:
\begin{equation*}
  \Re zM_r(z)\geq \min\{\Re z,\Re A_{r}^{-1}\}
\end{equation*}
and 
\begin{equation*}
   \norm{M_{r}(z)}\leq 1+\norm{z^{-1}A_{r}^{-1}}\leq \frac{1}{c\abs{z}}
\end{equation*}
for $\abs{z}>0$. We obtain
\begin{align*}
   \Re A_r^{-1}&= \Re (a_1-a_2\mathcal{R}_r a_3 )^{-1} \\
   &\geq c \norm{(a_1-a_2\mathcal{R}_r a_3 )}^{-2} \geq c(\beta+\beta^2\sup_{0\leq r\leq 1}\norm{\mathcal{R}_r})^{-2}
\end{align*}
by
\Cref{le:realpart-bounded-from-below-implies-invertible}.

Since $(\mathcal{R}_r)_r$ is an approximation of unity, it follows from \Cref{le:sot-inverse-sequence-conv-crit} that $A^{-1}_r\to A^{-1}_0$ in the strong operator topology as $r \to 0$.
\begin{theorem}\label{thm:applcellmigration}
For all $\nu>0$, we have
\begin{multline*}
  \cl{\left[ \partial_t
      \begin{pmatrix}1 & 0 \\ 0 & 0\end{pmatrix}
      + \begin{pmatrix}0 & 0 \\ 0 & A_r^{-1}\end{pmatrix}
      + \begin{pmatrix} 0& \divn \\ \grad & 0\end{pmatrix}
    \right]}^{-1} \\
  \to
  \cl{\left[
      \partial_t \begin{pmatrix}1 & 0 \\ 0 & 0\end{pmatrix}
      + \begin{pmatrix}0 & 0 \\ 0 & A_0^{-1}\end{pmatrix}
      + \begin{pmatrix} 0 & \divn \\ \grad & 0\end{pmatrix}
    \right]}^{-1}
\end{multline*}
as $r\to 0$ in the weak operator topology of $\Lb\big(\Lp[\nu]{2}(\R;\Lp{2}(\Omega)^{n+1})\big)$.
\end{theorem}
\begin{proof}
  Considering the Rellich--Kondrachov theorem and the above discussion,
  this immediately follows from \Cref{le:components-holomorphic,le:realpart-bounded-from-below-implies-invertible,le:dohnal}, \Cref{cor:sotnlH}
  and \Cref{th:second-main-theorem}.
\end{proof}
\begin{remark}
  In \cite[Theorem 5.1.3]{W16h}, one can show -- even in the non-autonomous case -- that the solution operators even converge in the strong operator topology. The example is merely presented to have a nonlocal example at hand.

  In the case $n=3$, note that the convergence assumptions of the above theorem can be weakened. We particularly refer to the example in \cite{Wa18} showing that if, additionally, $a_1$ is replaced by an $\mathrm{H}$-converging sequence $(a_{1,k})$ with limit $a_1$, the resulting sequence
 \begin{equation*}
   (a_{1,k}-a_2\mathcal{R}_{1/k} a_3 )_k^{-1}
 \end{equation*}converges to $   (a_{1}-a_2\mathcal{R}_{0} a_3 )^{-1}$ in $\hyperref[def:nHc]{\tau(\rg,\rrn)}$.
\end{remark}

\subsection{A homogenisation problem for scalar piezo-electricity}\label{subsec:hompie}

In this section, we consider a classical homogenisation problem in order to showcase the applicability for rapidly oscillating albeit local coefficients. Again, we refer to the example in \cite{Wa18} for more sophisticated situations. Here, we follow the model description of piezo-electro-magnetism from \cite{P17}. Note, that we treat homogeneous Dirichlet boundary conditions throughout and for ease of readability we simplify the case to scalar elastic waves. The rationale for 3-dimensional elastic waves can be dealt with similarly. Let $\Omega\subseteq \R^3$ be a bounded, weak Lipschitz domain with continuous boundary. Additionally, assume that $\Omega$ is topologically trivial (recall \Cref{ex:Htop} for the Helmholtz decomposition). We adopt the notation rolled out in \cite{P17} and consider the evolutionary equation $(\partial_t M_0+M_1+A)U=F$ with the following setting
\begin{align*}
  M_0 &\coloneqq
        \begin{pmatrix}
          1 & 0 & 0 & 0 \\
          0 & C^{-1} & C^{-1} e & 0 \\
          0 & e^* C^{-1} & \varepsilon + e^* C^{-1}e & 0 \\
          0 & 0 & 0 & \mu
        \end{pmatrix},
  &
    M_1 &\coloneqq
          \begin{pmatrix}
            0 & 0 & 0 & 0 \\
            0 & 0 & 0 & 0 \\
            0 & 0 & \sigma & 0 \\
            0 & 0 & 0 & 0
          \end{pmatrix},
  \\
  A &\coloneqq
      \begin{pmatrix}
        0 & -\div & 0 & 0 \\
        -\gradn  & 0 & 0 & 0 \\
        0 & 0 & 0 & -\rot \\
        0 & 0 & \rotn & 0
      \end{pmatrix},
\end{align*}
where $C$, $e$, $\mu$, $\varepsilon$, $\sigma$ are operators in $\Lb(\Lp{2}(\Omega)^3)$, of which $C$, $\mu$ and $\varepsilon$ are self-adjoint and non-negative, and $\rotn\coloneqq \cl{\rot|_{\Cc(\Omega)^3}}$ is defined similarly to $\divn$ before. Well-posedness in $\Lp[\nu]{2}(\R;\Lp{2}(\Omega)^{10})$ can be guarenteed by~\cite{P17}, if for some $c,d>0$ and $\nu_{0}\geq 0$ we have
\begin{equation*}
  C \geq 1/d,\quad \mu \geq c \quad\text{and}\quad \nu\varepsilon +\Re \sigma\geq c
\end{equation*}
for all $\nu>\nu_{0}$. Additionally asking for
\begin{equation*}
  C^{-1} \geq c,\quad \mu^{-1} \geq 1/d \quad\text{and}\quad \Re\big((\varepsilon + \sigma/z)^{-1}\big)\geq 1/d
\end{equation*}
for $\Re z>\nu_{0}$, we analogously obtain~\eqref{eq:alternative-boundedness-sec-main-th}.
In order to address the homogenisation problem, we consider bounded sequences $(C_n)_n$, $(e_n)_n$, $(\mu_n)_n$, $(\varepsilon_n)_n$, $(\sigma_n)_n$
where we assume the same self-adjointness, non-negativity and positive-definiteness conditions as before for $ C, \varepsilon, \mu,  \sigma$. The positive-definiteness constants $c,d$ and $\nu_{0}$ are supposed to be independent of $n$.

The operator $A$ induces the following decomposition (see \Cref{ex:Htop}) of the space $\mathcal{H} = \Lp{2}(\Omega)^{10}$:
\begin{equation*}
  \mathcal{H} = \Lp{2}(\Omega)^{10}
  = (\underbrace{\set{0} \oplus \rr \oplus \rgn \oplus \rg}_{=\mathrlap{\ker A = \mathcal{H}_{0}}})
  \oplus (\underbrace{\Lp{2}(\Omega) \oplus \rgn \oplus \rr \oplus \rrn}_{=\mathrlap{\ran A = \mathcal{H}_{1}}}).
\end{equation*}
Note that the assumption that $\Omega$ is topologically trivial guarentees that $\mathcal{H}_{D}(\Omega) = \set{0} = \mathcal{H}_{N}(\Omega)$ and therefore these spaces do not appear in the Helmholtz decomposition.

The application of \Cref{th:second-main-theorem} now reads as follows:

\begin{theorem}\label{thm:piezo-electr-hom}
  Assume that for $\nu>\nu_{0}$ and for all $z\in \C_{\Re>\nu}$
  \begin{equation*}
    \begin{pmatrix}C_n^{-1} & C_n^{-1} e_n \\
      e_n^{\ast} C_n^{-1} & \varepsilon_{n} + e_n^* C_n^{-1}e_n +z^{-1}\sigma_n
    \end{pmatrix}
  \end{equation*}
  converges to some $Z(z^{-1})$ in $\hyperref[def:nHc]{\tau(\rr \oplus \rgn, \rgn\oplus \rr)}$ and $\mu_n\to \mu$ in $\hyperref[def:nHc]{\tau(\rgn,\rr)}$ as $n\to\infty$. Then
  \begin{equation*}
    M\colon z\mapsto
    \begin{pNiceArray}{cw{c}{0.5cm}cc}
      1 &  0 & 0 & 0\\
      0 & \Block[c]{2-2}{Z(z^{-1})} & & 0 \\
      0 &  &  & 0 \\
      0 & 0 & 0 &\mu
    \end{pNiceArray}
    \in \Lb\big(\Lp{2}(\Omega)\oplus \Lp{2}(\Omega)^6\oplus \Lp{2}(\Omega)^3\big)
  \end{equation*}
  is a material law that satisfies $\Re zM(z)\geq c'$ for all $z\in \C_{\Re>\nu}$ and a suitable $c'>0$. Moreover,
  \begin{multline*}
    \cl{\left[\partial_t
        \begin{pmatrix}
          1 & 0 & 0 & 0 \\
          0 & C_{n}^{-1} & C_{n}^{-1} e_{n} & 0 \\
          0 & e_{n}^* C_{n}^{-1} & \varepsilon_{n} + e_{n}^* C_{n}^{-1}e_{n} & 0 \\
          0 & 0 & 0 & \mu_{n}
        \end{pmatrix}
        +
        \begin{pmatrix}
          0 & 0 & 0 & 0 \\
          0 & 0 & 0 & 0 \\
          0 & 0 & \sigma_{n} & 0 \\
          0 & 0 & 0 & 0
        \end{pmatrix}
        + A\right]}^{-1} \\
    \to \cl{[\partial_t M(\partial_t) + A]}^{-1}
  \end{multline*}
  in the weak operator topology in $\Lb(\Lp[\nu]{2}(\R;\Lp{2}(\Omega)^{10})$.
\end{theorem}

\begin{proof}
  Considering the Rellich--Kondrachov theorem, the compact embeddings of $\dom(\rot)\cap\rrn$ and $\dom(\rotn)\cap\rr$ as Hilbert spaces into $\Lp{2}(\Omega)^{3}$ and the above discussion, the claim follows from \Cref{th:second-main-theorem} in
  combination with \Cref{re:second-main-theorem}.
\end{proof}

\begin{remark}
  It is desirable to obtain a more explicit formula for the limit expression in \Cref{thm:piezo-electr-hom} if more structural assumptions on the coefficients and the couplings are at hand. In fact, in a slightly different situation this is done in \cite{F83}. Thus, at least for periodic, highly oscillatory coefficents one can anticipate the existence of a limit. The particular computation of which, however, will be left to future research.
\end{remark}

\section{Conclusion}\label{sec:con}

In this paper, we have defined a topology on holomorphic, operator-valued functions
(\Cref{def:main-topology}) and provided a compactness result (\Cref{th:main-theorem}). Moreover, we have identified a continuity statement related to the resolvent of a skew-selfadjoint operator with compact resolvent outside its kernel that, together with the introduced topology, yields a convergence result (\Cref{th:second-main-theorem}) that has applications to (abstract, nonlocal) homogenisation problems for evolutionary equations and can be easily applied to a class of nonlocal equations as well as to homogenisation problems for systems of time-dependent partial differential equations.

\appendix

\section{}\label{app:app}

\begin{remark}\label{re:compact-convergence-topology-properties}
  Whenever we consider the holomorphic functions $\Hol(U,\C)$ for an open $U\subseteq\C$, we endow this space with the topology of compact convergence. That means, a sequence
  in $\Hol(U,\C)$
  converges if and only if it converges uniformly on every compact subset of $U$. One can (cf.~\cite{Arens46} and~\cite[Thm.~3.4.16]{Engelking89}) explicitly construct a complete and separable metric that induces this topology,
  i.e., $\Hol(U,\C)$ is Polish. Note that (obviously) both addition and scalar multiplication are continuous w.r.t.\ this topology, i.e., $\Hol(U,\C)$ is a topological vector space.
\end{remark}

\begin{lemma}\label{le:components-holomorphic}
  Let $U \subseteq \C$ be an open set and let $\mathcal{H}$ be a Hilbert space that can be orthogonally decomposed into $\mathcal{H} = \mathcal{H}_{0} \oplus \mathcal{H}_{1}$.
  Then, the block operator
  \begin{equation*}
    M=\begin{pmatrix}
      M_{00} & M_{01} \\
      M_{10} & M_{11}
    \end{pmatrix}\colon U \to \mathcal{H}^{\mathcal{H}}
  \end{equation*}
maps to $\Lb(\mathcal{H})$ if and only if
each block entry $M_{ij}\colon U\to \mathcal{H}_{i}^{\mathcal{H}_{j}}$ maps
to $\Lb(\mathcal{H}_{j},\mathcal{H}_{i})$. In that case, $M\colon U \to\Lb(\mathcal{H})$ is
holomorphic if and only if each block entry $M_{ij}\colon U\to \Lb(\mathcal{H}_{j},\mathcal{H}_{i})$ is holomorphic.
\end{lemma}

\begin{proof}
  For $v\in\mathcal{H}$ with $\norm{v}_{\mathcal{H}}\leq 1$ and its unique decomposition
  $v=v_{0}+v_{1}$, the Pythagorean theorem yields $\norm{v_{0}}_{\mathcal{H}},
  \norm{v_{1}}_{\mathcal{H}}\leq 1$. Once again applying the Pythagorean theorem, we obtain
  \begin{align*}
    \norm{M(z)v}^{2}_{\mathcal{H}}&=\norm{M_{00}(z)v_{0}+M_{01}(z)v_{1}}^{2}_{\mathcal{H}}
    +\norm{M_{10}(z)v_{0}+M_{11}(z)v_{1}}^{2}_{\mathcal{H}}\\
    &\leq (\norm{M_{00}(z)}+\norm{M_{01}(z)})^{2}+(\norm{M_{10}(z)}+\norm{M_{11}(z)})^{2}
  \end{align*}
  for $z\in U$. This shows
  \begin{equation}
    \norm{M(z)}\leq \norm{M_{00}(z)}+\norm{M_{01}(z)}+\norm{M_{10}(z)}+\norm{M_{11}(z)}\text{.}
    \label{eq:full-vs-block-matrix-norm}
  \end{equation}
  Conversely, assume $v_{0}\in\mathcal{H}_{0}$ with $\norm{v_{0}}\leq 1$. Then, the Pythagorean
  theorem yields
  \begin{equation*}
    \norm{M_{00}(z)v_{0}}^{2}_{\mathcal{H}}\leq \norm{M_{00}(z)v_{0}}^{2}_{\mathcal{H}}+\norm{M_{10}(z)v_{0}}^{2}_{\mathcal{H}}=
    \norm{M(z)v_{0}}^{2}_{\mathcal{H}}\leq \norm{M(z)}^{2}
  \end{equation*}
  for $z\in U$. After similar calculations for the other block entries, we get
  \begin{equation}\label{eq:block-vs-full-matrix-norm}
    \max\big(\norm{M_{00}(z)},\norm{M_{01}(z)},\norm{M_{10}(z)},\norm{M_{11}(z)}\big) \leq \norm{M(z)}\text{.}
  \end{equation}
  Inequalities~\eqref{eq:full-vs-block-matrix-norm} and~\eqref{eq:block-vs-full-matrix-norm} immediately prove the claimed statements.
\end{proof}

\begin{lemma}\label{le:product-holomorphic}
  Let $U \subseteq \C$ be an open set and let $\mathcal{H}$ be a Hilbert space. If $M, N \colon U \to \Lb(\mathcal{H})$ are holomorphic, then also the product $MN$ is holomorphic with derivative $MN'+M'N$.
\end{lemma}

\begin{proof}
  Note that the multiplication in $\Lb(\mathcal{H})$ is a continuous operation. Hence,
  \begin{align*}
    \MoveEqLeft[7]
    \lim_{w \to z} \frac{M(z)N(z) - M(w)N(w)}{z-w} \\
    &=\lim_{w \to z} \frac{M(z)N(z) - M(z)N(w) + M(z)N(w) - M(w)N(w)}{z-w} \\
    &=\lim_{w \to z} \frac{M(z)N(z) - M(z)N(w)}{z-w} + \lim_{w \to z}\frac{M(z)N(w) - M(w)N(w)}{z-w} \\
    &= M(z) \lim_{w \to z} \frac{N(z) - N(w)}{z-w} + \lim_{w \to z}\frac{M(z) - M(w)}{z-w} N(w) \\
    &= M(z) N'(z) + \lim_{w \to z}\frac{M(z) - M(w)}{z-w} \lim_{w \to z}N(w) \\
    &= M(z) N'(z) + M'(z)N(z),
  \end{align*}
  which implies that $MN$ is complex differentiable.
\end{proof}

\begin{lemma}\label{le:inverse-holomorphic}
  Let $U \subseteq \C$ be an open set and let $\mathcal{H}$ be a Hilbert space. If a holomorphic function $M\colon U \to \Lb(\mathcal{H})$ is such that $M(z)$ is invertible for every $z \in U$, then $M(\cdot)^{-1}\colon U \to \Lb(\mathcal{H})$ is also holomorphic with derivative $-M(\cdot)^{-1} M' M(\cdot)^{-1}$.
\end{lemma}

\begin{proof}
  Note that $A \mapsto A^{-1}$ is continuous on $\dset{A \in \Lb(\mathcal{H})}{A \text{ is invertible}}$ by \cite[Thm.~10.12]{Ru91} and the multiplication in $\Lb(\mathcal{H})$ is a continuous operation. Hence,
  \begin{align*}
    \lim_{w \to z}\frac{M(z)^{-1} - M(w)^{-1}}{z-w}
    &= \lim_{w \to z} \frac{M(z)^{-1}(M(w) - M(z)) M(w)^{-1}}{z-w} \\
    &=  M(z)^{-1}\lim_{w \to z}\frac{-(M(z) - M(w))}{z-w} \lim_{w \to z} M(w)^{-1} \\
    &= - M(z)^{-1} M'(z) M(z)^{-1},
  \end{align*}
  which implies that $M(\cdot)^{-1}$ is complex differentiable.
\end{proof}

\begin{theorem}[Montel's theorem {\cite[Thm.~14.6]{Ru87}}]\label{th:montels-theorem}
  Let $U \subseteq \C$ be open and $S \subseteq \Hol(U,\C)$. Then $S$ is relatively compact (also called normal), if and only if, $S$ is locally uniformly bounded, i.e., for all $K \subseteq U$ compact there exists a $C_{K} > 0$ such that
  \begin{equation*}
    \sup_{z \in K, f \in S} \abs{f(z)} \leq C_{K}.
  \end{equation*}
\end{theorem}

\begin{corollary}\label{th:Hol-U-ball-compact}
  $\Hol(U,\cl{\ball}_{r}(0))$ is compact, where $\cl{\ball}_{r}(0)$ is the closed ball with radius $r \geq 0$ in $\C$.
\end{corollary}

\begin{proof}
  By \Cref{th:montels-theorem} (Montel's theorem) we conclude that $\Hol(U,\cl{\ball}_{r}(0))$ is relatively compact in $\Hol(U,\C)$.
  We finish the proof by showing the closedness of $\Hol(U,\cl{\ball}_{r}(0))$:
  Let $(f_{n})_{n\in\N}$ be a sequence in $\Hol(U,\cl{\ball}_{r}(0))$ that converges to $f \in \Hol(U,\C)$, i.e., for all $K \subseteq U$ compact we have
  \begin{align*}
    \sup_{z \in K} \abs{f_{n}(z) - f(z)} \to 0.
  \end{align*}
  In particular $f_{n}(z)$ converges to $f(z)$ in $\C$. Since $\abs{f_{n}(z)} \leq 1$ and limits preserve inequalities, we conclude $\abs{f(z)} \leq 1$.
\end{proof}

The following theorem is a small adaption of \cite[Prop.~6.1]{Ba85}. We just regard $U \subseteq \C$ instead of the more general case $U \subseteq \mathcal{Y}$ for a normed vector space $\mathcal{Y}$.
\begin{theorem}\label{th:weak-is-strong-holomorphy}
  Let $\mathcal{X}$ be a Banach space and $U \subseteq \C$ open. If $\Psi \subseteq \mathcal{X}^{\prime}$ has the following property
  \begin{equation*}
    W \subseteq \mathcal{X}\ \text{is bounded}
    \quad\Leftrightarrow\quad
    \psi(W) \subseteq \C\ \text{is bounded}\ \forall \psi \in \Psi,
  \end{equation*}
  then the following statements are equivalent:
  \begin{enumerate}
    \item $f \in \Hol(U,\mathcal{X})$,
    \item $\psi \circ f \in \Hol(U,\C)$ for all $\psi \in \Psi$.
  \end{enumerate}
\end{theorem}

\begin{lemma}[{\cite[Prop.~6.2.3]{SeTrWa22}}]\label{le:realpart-bounded-from-below-implies-invertible}
  Let $\mathcal{H}$ be a Hilbert space and $A \in \Lb(\mathcal{H})$ such that $\Re A \geq c > 0$. Then, $A^{-1} \in \Lb(\mathcal{H})$ with $\norm{A^{-1}} \leq \frac{1}{c}$
  and $\Re A^{-1}\geq c\norm{A}^{-2}$.
\end{lemma}

\begin{lemma}[{{\cite[Lemma 3.9]{DTM23}}}]\label{le:dohnal}
  Let $\mathcal{H}$ be a Hilbert space that can be orthogonally decomposed into $\mathcal{H} = \mathcal{H}_{0} \oplus \mathcal{H}_{1}$.
  Consider an operator $T\in\Lb(\mathcal{H})$ in his block form $(T_{ij})_{i,j\in \{0,1\}}$ (cf.\ \Cref{le:components-holomorphic}).
  If we have $\Re T\geq d$ for some $d>0$,
  then $\Re T_{11}\geq d$ and $\Re (T_{00}-T_{01}T_{11}^{-1}T_{10})\geq d$ follow.
\end{lemma}

\begin{proof}
  Let $\varphi_1\in \mathcal{H}_1$. Then,
  \begin{equation*}
    \Re \scprod{T_{11}\varphi_1}{\varphi_1}_{\mathcal{H}_{1}}
    =\Re \scprod*{T\begin{pmatrix} 0 \\ \varphi_1\end{pmatrix}}{\begin{pmatrix} 0 \\ \varphi_1\end{pmatrix}}_{\mathcal{H}}
    \geq d \scprod*{\begin{pmatrix} 0 \\ \varphi_1\end{pmatrix}}{\begin{pmatrix} 0 \\ \varphi_1\end{pmatrix}}_{\mathcal{H}}
    = d \scprod{\varphi_1}{\varphi_1}_{\mathcal{H}_{1}}.
  \end{equation*}
  By \Cref{le:realpart-bounded-from-below-implies-invertible} it follows that $T_{11}$ is invertible. For the accretivity of the second expression, one quickly checks the relation $R=Q^*TQ$, where
  \begin{equation*}
    Q \coloneqq \begin{pmatrix} 1 & 0 \\ -(T_{01}T_{11})^* & 1 \end{pmatrix}
    \quad\text{and}\quad
    R \coloneqq \begin{pmatrix} T_{00}-T_{01}T_{11}^{-1}T_{10} & 0 \\ T_{10}-T_{11}(T_{11}^{-1})^*T_{01} & T_{11}\end{pmatrix}.
  \end{equation*}
  Next, we let $\varphi_0\in \mathcal{H}_0$, $S\coloneqq T_{00}-T_{01}T_{11}^{-1}T_{10}$ and compute
  \begin{align*}
    \Re \scprod{S\varphi_0}{\varphi_0}_{\mathcal{H}_{0}}
    &= \Re \scprod*{R \begin{pmatrix} \varphi_0 \\ 0 \end{pmatrix}}{\begin{pmatrix} \varphi_0 \\ 0 \end{pmatrix}}_{\mathcal{H}} \\
    &= \Re \scprod*{TQ\begin{pmatrix} \varphi_0 \\ 0 \end{pmatrix}}{Q\begin{pmatrix} \varphi_0 \\ 0 \end{pmatrix}}_{\mathcal{H}} \\
    &\geq d \scprod*{Q\begin{pmatrix} \varphi_0 \\ 0 \end{pmatrix}}{Q\begin{pmatrix} \varphi_0 \\ 0 \end{pmatrix}}_{\mathcal{H}}
      \geq d \scprod{\varphi_0}{\varphi_0}_{\mathcal{H}_{0}}. \qedhere
  \end{align*}
\end{proof}

\begin{lemma}[{\cite[Prop.~13.1.4]{SeTrWa22}}]\label{le:sot-inverse-sequence-conv-crit}
  Let $\mathcal{H}$ be a Hilbert space and $(T_{n})_{n\in\N}$ a boundedly invertible sequence
  in $\Lb(\mathcal{H})$ such that $\sup_{n\in\N}\norm{T_{n}^{-1}}<\infty$. If $(T_{n})_{n\in\N}$
  converges to a $T\in \Lb(\mathcal{H})$ w.r.t.\ the strong operator topology and if $T$ has
  dense range, then $T$ is boundedly invertible and $(T^{-1}_{n})_{n\in\N}$ converges
  to $T^{-1}$ w.r.t.\ the strong operator topology.
\end{lemma}

\begin{lemma}[{\cite[Prop.~6.3.1]{SeTrWa22}}]\label{le:ubd-realpart-bounded-from-below-implies-invertible}
  Let $\mathcal{H}$ be a Hilbert space and $A\colon\dom (A)\subseteq\mathcal{H}\to\mathcal{H}$ densely defined and closed with
  $\dom (A^{\ast})\subseteq\dom (A)$. If $\Re \scprod{A\varphi}{\varphi}_{\mathcal{H}} \geq c > 0$ holds for all $\varphi\in\dom (A)$,
  then $A^{-1} \in \Lb(\mathcal{H})$ and $\norm{A^{-1}} \leq \frac{1}{c}$.
\end{lemma}

\bibliographystyle{alphaurl}
\bibliography{bibfile}{}

\end{document}